\documentclass[a4paper,11pt]{amsart}
\usepackage{amsmath,amssymb,amsfonts,amsthm,exscale,calc}
\usepackage{latexsym,textcomp}
\usepackage{dsfont}
\usepackage{a4wide}
\usepackage{dsfont}
\usepackage[german, english]{babel}
\usepackage{hyperref}
\usepackage{enumerate}
\usepackage[utf8]{inputenc}

\newtheorem{lemma}{Lemma}[section]
\newtheorem{theorem}[lemma]{Theorem}

\newtheorem{proposition}[lemma]{Proposition}
\newtheorem{corollary}[lemma]{Corollary}

\theoremstyle{definition}
\newtheorem{definition}[lemma]{Definition}

\newtheorem*{remark}{Remark}
\newtheorem*{assumption}{Assumption}

\numberwithin{equation}{section}
\newcommand{\comment}[1]{}

\newcommand{\cE}{{\mathcal E}}

\newcommand{\R}{{\mathbb R}}

\newcommand{\N}{{\mathbb N}}

\newcommand{\lpm}{{L^p(X,\mu)}}

\newcommand{\as}[1]{\langle #1\rangle}
\newcommand{\av}[1]{\left\Vert #1\right\Vert}
\newcommand{\aV}[1]{\left\Vert #1\right\Vert}

\newcommand{\Hm}[1]{\leavevmode{\marginpar{\tiny%
$\hbox to 0mm{\hspace*{-0.5mm}$\leftarrow$\hss}%
\vcenter{\vrule depth 0.1mm height 0.1mm width \the\marginparwidth}%
\hbox to 0mm{\hss$\rightarrow$\hspace*{-0.5mm}}$\\\relax\raggedright
#1}}}

\newcommand{\Ee}{\mathcal{E}_{\rm e}}

\DeclareMathOperator*{\esssup}{ess\,sup}
\DeclareMathOperator*{\essinf}{ess\,inf}

\begin{document}

\title{(Weak) Hardy and Poincaré inequalities and criticality theory}

\author{Marcel Schmidt}
\address{Mathematisches Institut, Universität Leipzig, 04109 Leipzig, Germany.} \email{marcel.schmidt@math.uni-leipzig.de}

\maketitle

\begin{abstract}
 In this paper we study Hardy and Poincaré inequalities and their weak versions for quadratic forms satisfying the first Beurling-Deny criterion.  We employ these inequalities to establish a criticality theory for such forms. 
\end{abstract}

%\subjclass[2000]{Primary 39A12; Secondary 58J35}
%\date{\today}

%\tableofcontents

%%%%%%%%%%%%%%%%%%%%%%%%%%%%%%%%%%%%%%%%%%%%%%%%%%%%%%%%%
%%%%%%%%%%%%%%%%%%%%%%%%%%%%%%%%%%%%%%%%%%%%%%%%%%%%%%%%%

% \section{Introduction}
% 
% 
% In this note we prove the following Perron-Frobenius / Allegretto-Piepenbrink type result. For $1 < p < \infty$ let $T \colon L^p(X,\mu) \to L^p(Y,\nu)$ be a positive kernel   kernel operator. Then there exists $0 \neq f \in L^0_+(X,\mu)$ such that
% %
% $$T^* (T f)^{p-1} \leq C_P f^{p-1}.$$
% %
% Note that by abstract nonsense $C_p$ is the best possible constant. In particular, we investigate the case $p = 2$ for semigroups and resolvents of forms satisfying the first Beurling-Deny criterium.

\section{Introduction}

Let $q \colon L^2(X,\mu) \to [0,\infty]$ be a closed quadratic form satisfying the first Beurling-Deny criterion. A typical example is the quadratic form of a (discrete) Schrödinger operator, see e.g. \cite{PT,Tak1,KPP} or a Dirichlet form, see e.g. \cite{FOT}. We deal with inequalities of the form 
\begin{align}\label{inequality:fundamental}
 \int_X f^2 w d\mu \leq \alpha(r) q(f) + r \Phi(f)^2, \quad f \perp_w \ker q,\, r > 0. \tag{$\Diamond$}
\end{align}
Here, $w \colon X \to (0,\infty)$, $\alpha \colon (0,\infty) \to (0,\infty)$ is decreasing, $\Phi \colon L^2(X,\mu) \to [0,\infty]$ is a sublinear functional and the symbol $\perp_w$ indicates that orthogonality is considered in $L^2(X,w\mu)$. If $\ker q = \{0\}$, Inequality~\eqref{inequality:fundamental} is called {\em weak Hardy inequality} and if $\ker q \neq \{0\}$, Inequality~\eqref{inequality:fundamental} is called {\em weak Poincaré inequality}. In the case $\Phi = 0$, the function $\alpha$ becomes a constant and Inequality~\eqref{inequality:fundamental} is referred to as {\em Hardy inequality} if $\ker q = \{0\}$, respectively  {\em Poincaré inequality} if $\ker q \neq \{0\}$. 

The goal of this paper is to provide abstract criteria for (weak) Hardy/Poincaré inequalities with respect to the sublinear functionals $\Phi = 0$ and $\Phi(f) = \av{f/h}_\infty$, where $h \colon X\to (0,\infty)$. We study which $w$ are eligible but do not aim at giving explicit bounds for $\alpha$ if $\Phi \neq 0$. Moreover, we apply these inequalities to establish a criticality theory for general forms satisfying the first Beurling-Deny criterion.

It turns out that if one replaces $ \int f^2 w d\mu$ on the left side of Inequality~\ref{inequality:fundamental} by $(\int f w d\mu)^2$, then for generic sublinear $\Phi$ there exists a function $\alpha$ such this inequality holds, see Theorem~\ref{theorem:very weak}. For forms with the first Beurling-Deny criterion and $\ker q = \{0\}$, this always leads to a weak Hardy inequality with respect to $\Phi(f) = \av{f/h}_\infty$ as long as $h \in L^2(X,\mu)$ and $w \in L^2(X,h^2\mu)$, see Theorem~\ref{theorem:weak hardy}. Thus, one can say that weak Hardy inequalities hold generically for forms with the first Beurling-Deny criterion and trivial kernel. 

Forms satisfying a Hardy inequality are called {\em subcritical}  and there is a large amount of literature on subcriticality, see e.g. (and references therein) \cite{Mur,Pin,Pins} for elliptic operators with real coefficients, \cite{Fit,Tak1,TU} for generalized Schrödinger forms (Dirichlet form plus potential term) and \cite{KPP} for discrete Schrödinger operators. It turns out that weak Hardy inequalities can be employed to study subcriticality.  We give a comprehensive characterization in Theorem~\ref{theorem:subcriticality}, which should cover most previous results (for quadratic forms), simplifies/unifies proofs and also gives some new insights.  

If $q$ is irreducible, in the situations where subcriticality is well understood (see the mentioned references on subcriticality), there is a dichotomy. Either $q$ is subcritical or there exists a sequence  $(\varphi_n)$  with $q(\varphi_n) \to 0$ that converges pointwise to a strictly positive function $h$.  In this second case, $q$ is called {\em critical} and the function $h$ is unique up to multiplication by a constant, the so-called {\em Agmon ground state} of $q$. Our discussion in Theorem~\ref{theorem:subcriticality} shows that this dichotomy may fail in general. It becomes a trichotomy and the third case (besides criticality and subcriticality) happens if and only if the form $q$ does not possess an excessive function, see Corollary~\ref{corollary:trichotomy}. So far we do not have a concrete example for this case. The reason is that in the situations where subcriticality was studied previously, the corresponding semigroups  are semigroups of kernel operators and we show in Appendix~\ref{section:appendix} that irreducible semigroups of kernel operators always admit excessive functions.  Along the way we give a partial answer to a question of Schep \cite{Schep} on more general positive operators on $L^p$-spaces.

Weak Poincaré inequalities were introduced in \cite{RW} for conservative Dirichlet forms on finite measure spaces to study the rate of convergence of their semigroups to equilibrium when there is no spectral gap, i.e., when they do not satisfy a Poincaré inequality with $w = 1$.  Similarly, the weak Hardy inequalities mentioned above can be employed to establish the rate of convergence to $0$ of semigroups coming from forms with trivial kernel, see the remark after Theorem~\ref{theorem:weak hardy}. 

A closed quadratic form $q$ on $L^2(X,\mu)$ with the first Beurling-Deny criterion can be extended to a lower semicontinuous quadratic form $q_e$ on $L^0(X,\mu)$, the so-called extended form. With this at hand the Agmon ground state (if it exists) is an element of the kernel of $q_e$ and criteria for subcriticality can be formulated conveniently in terms of $q_e$. One of the observations is that subcriticality is equivalent to the domain of $q_e$ being a Hilbert space, a fact which is well-known for Dirichlet forms (where subcriticality is called transience and the domain of the extended form is the extended Dirichlet space), see e.g. \cite[Section~1.6]{FOT}. Moreover, it is known for critical Dirichlet forms (usually called recurrent Dirichlet forms)  that the quotient of the extended Dirichlet space modulo constants is a Hilbert space if  a Poincaré inequality holds, see e.g. \cite[Section~4.8]{FOT}. We show in Theorem~\ref{theorem:completeness extended space critical case} that some sort of converse holds in our setting: In the critical case completeness of the domain of $q_e$ modulo the kernel of $q_e$ is equivalent to a weak Poincaré inequality. This observation seems to be new. An example from \cite{RW} shows that there are irreducible conservative Dirichlet forms without weak Poincaré inequality and as a consequence we obtain that their extended Dirichlet space modulo constants (i.e. the domain of the extended form modulo its kernel) is not complete, see Corollary~\ref{coro:non complete}. To the best of our knowledge it is a new observation that such forms exist. 

At the heart of the considerations regarding the completeness of the domain of the extended form lies the following observation: If $q$ has an excessive function,  weak Hardy/Poincaré inequalities are equivalent to the continuity of the embedding of (a quotient) of the domain of the extended form equipped with the norm coming from $q_e$ into (a quotient) of $L^0(X,\mu)$ equipped with (the quotient topology of) the topology of local convergence in measure. That this continuity is equivalent to completeness of the the domain of the extended form is a consequence of the abstract Theorem~\ref{theorem:2 imply the third}, which is taken from \cite{Schmi3}.

For technical reasons and in order to understand the third case besides criticality and subcriticality in the mentioned trichotomy better, we  introduce another functional $q^+$ on $L^+(X,\mu)$ (the cone of $[0,\infty]$-valued $\mu$-a.e. defined functions), which is a lower semicontinuous extension of $q_e$. If $q$ is irreducible, the mentioned trichotomy then reads either $\ker q^+ = \{0\}$ (subcriticality) or $\ker q^+ = \{\lambda h \mid \lambda \in  [0,\infty]\}$ (criticality, with Agmon ground state $h$) or  $\ker q^+ = \{0,\infty\}$.

As discussed above, besides new insights, our approach to criticality theory also yields simplified/unified proofs for facts,  which are known for Dirichlet forms and some more general situations. We lay out the material such that our  proofs do not depend on known aspects of criticality theory, for otherwise they would not really simplify/unify arguments. To this end, we discuss properties of forms with the first Beurling-Deny criterion in Section~\ref{section:BD} in some detail.  Among other things we give a short independent proof for the existence of the extended forms $q_e$, see Proposition~\ref{prop:existence of extended form}, which otherwise could be deduced from \cite{Schmu2}. 

{\em Acknowledgments:} The author thanks Matthias Keller, Felix Pogorzelski and Yehuda Pinchover for sharing their knowledge on criticality theory and for listening to his ideas on the matter during a stay at the Technion in Haifa.   Moreover, he thanks Anton Schep for discussions on positive operators. 

%It turns out that if $\ker q = \{0\}$, then weak Hardy inequalities hold generically, see 

% 
% Weak Poincaré inequalities were introduced in \cite{RW} to study the rate of convergence of semigroups of conservative Dirichlet forms without spectral gap (i.e. without Poincaré inequality).

\section{Preliminaries}

\subsection{Closed quadratic forms on $L^2(X,\mu)$ and $L^0(X,\mu)$}

Throughout this text $(X,\mathfrak{A},\mu)$ is a  $\sigma$-finite measure space. By $L^+(X,\mu)$ we denote the quotient  of the space of measurable functions $f \colon X \to [0,\infty]$ with respect to equivalence $\mu$-a.e. and by $L^0(X,\mu)$ we denote the quotient  of the space of measurable functions $f \colon X \to \R$ with respect to equivalence $\mu$-a.e. For $f \in L^+(X,\mu)$ we write $f > 0$ if $f(x) > 0$ for $\mu$-a.e. $x \in X$. For any sequence $(f_n)$ in $L^+(X,\mu)$ the pointwise a.e. defined functions $\liminf_{n \to \infty} f_n$ and $\limsup_{n \to \infty} f_n$ exist in $L^+(X,\mu)$.

We equip $L^0(X,\mu)$ and $L^+(m)$ with the topology of {\em local convergence in measure}. More precisely, we choose an integrable $\varphi \colon X \to (0,\infty)$ and define for $f \in L^0(X,\mu)$ the $F$-norm (which is not a norm)  
$$\av{f}_0 = \int_{X} (|f|\wedge 1)\varphi  d\mu.$$
It induces a metric $d$ on $L^0(X,\mu)$ by $d(f,g) = \av{f - g}_0$  and a metric $\rho$ on $L^+(X,\mu)$ by $\rho(f,g) = d(\arctan(f),\arctan(g))$ (with the convention $\arctan(\infty) = \pi/2$). A sequence $(f_n)$ in $L^0(X,\mu)$ or in $L^+(X,\mu)$ converges with respect to these metrics to a function $f$ if and only if any subsequence of $(f_n)$ has a subsequence converging to $f$ $\mu$-a.e.  In particular, the topologies generated by $d$ and $\rho$ coincide on $L^0_+(X,\mu) = L^0(X,\mu) \cap L^+(X,\mu)$. Both $L^0(X,\mu)$ and $L^+(X,\mu)$ are complete and  $L^0(X,\mu)$ is a topological vector space, i.e., addition and scalar multiplication are continuous. 

As usual, for $p \in [1,\infty]$ we denote by $L^p(X,\mu)$ the Lebesgue spaces of $p$-integrable real-valued functions with corresponding norm $\av{\cdot}_p$. The scalar product on $L^2(X,\mu)$ is denoted by $\as{\cdot,\cdot}$. By $L_+^p(X,\mu)$ we denote the cone of nonnegative functions in $L^p(X,\mu)$, $p \in \{0\} \cup [1,\infty]$. If $h \in L_+^0(X,\mu)$ is strictly positive, we let 
$$L^\infty_h(X,\mu) = \{f \in L^0(X,\mu) \mid f/h \in L^\infty(X,\mu)\}$$
and equip it with the norm $\av{f}_{h,\infty} = \av{f/h}_\infty$.

In this text we consider quadratic forms on the Hilbert space $L^2(X,\mu)$ and on the topological vector space  $L^0(X,\mu)$, and certain homogenous functionals on the cone $L^+(X,\mu)$. 

Let $p \in \{0,2\}$ and let  $q \colon L^p(X,\mu) \to [0,\infty]$  be a quadratic form with {\em domain}
$$D(q) = \{f \in L^p(X,\mu) \mid q(f) < \infty\}.$$
By polarization it induces a bilinear form on its domain, which we also denote by $q$. In this sense we have $q(f) = q(f,f)$ for $f \in D(q)$.

We call the quadratic form $q$ {\em closed} if $D(q)$ is complete with respect to the metric $d_q(f,g) = q(f-g)^{1/2} + \av{f-g}_p$. The topology induced by this metric on $D(q)$  is called the {\em form topology}. As is well known in the case $p=2$, closedness of a quadratic form is equivalent to  lower semicontinuity. According to Lemma~\cite[Lemma~A.4]{Schmi2}, the same is true for quadratic forms on any complete metrizable topological vector space. Since this is  important for us we formulate it as a lemma. 

\begin{lemma}
 Let $p \in \{0,2\}$ and let $q \colon L^p(X,\mu) \to [0,\infty]$ be a quadratic form. The following assertions are equivalent. 
 \begin{enumerate}[(i)]
  \item $q$ is closed. 
  \item $q$ is lower semicontinuous, i.e., for all $(f_n)$ in $L^p(X,\mu)$ the convergence $f_n \to f$ in $L^p(X,\mu)$ implies 
  $$q(f) \leq \liminf_{n \to \infty} q(f_n).$$
 \end{enumerate}
  If $p = 2$, both are equivalent to:
  \begin{enumerate}[(iii)]
   \item  $q$ is weakly lower semicontinuous, i.e., for all $(f_n)$ in $L^2(X,\mu)$ the convergence $f_n \to f$ weakly in $L^2(X,\mu)$ implies 
  $$q(f) \leq \liminf_{n \to \infty} q(f_n).$$
  \end{enumerate}
 \end{lemma}
\begin{remark}
 %The crucial point in the lemma is that we need not assume that $(f_n)$ or $f$ belong to $D(q)$. Below we will use this frequently.
What makes this lemma nontrivial is that in general $L^0(X,\mu)$ is not a locally convex vector space. Hence, one can not just use a modified version of the Hilbert space proof. 
\end{remark}

A quadratic form $q'$ on $\lpm$ is called an extension of $q$ if $D(q) \subseteq D(q')$ and $q' = q$ on $D(q)$. We call $q$ {\em closable} if it possesses a closed extension. As for closedness the standard characterization of closability extends from the Hilbert space case to $L^0(X,\mu)$, see \cite[Lemma~A.3]{Schmi2}. Again there is some difficulty because $L^0(X,\mu)$ is not locally convex.

\begin{lemma}\label{lemma:closability}
 Let $p \in \{0,2\}$ and let $q \colon L^p(X,\mu) \to [0,\infty]$ be a quadratic form. The following assertions are equivalent. 
 \begin{enumerate}[(i)]
  \item $q$ is closable. 
  \item $q$ is lower semicontinuous on its domain, i.e., for all $(f_n)$ in $D(q)$  and $f \in D(q)$ the convergence $f_n \to f$ in $L^p(X,\mu)$ implies 
  $$q(f) \leq \liminf_{n \to \infty} q(f_n).$$
  \end{enumerate}
  In this case, $q$ possesses a smallest closed extension $\bar q \colon L^p(X,\mu) \to [0,\infty]$, which is given by
  $$\bar q (f) = \begin{cases}
                  \lim\limits_{n \to \infty} q(f_n)  &\text{if there ex. $q$-Cauchy sequence } (f_n) \text{ with } f_n \to f \text{ in }L^p(X,\mu)\\
                  \infty &\text{else}
                 \end{cases}.
$$
\end{lemma}

% \begin{lemma}\label{lemma:weak convergence}
%   Let $p \in \{0,2\}$ and let $q \colon L^p(X,\mu) \to [0,\infty]$ be a closed quadratic form. If $(f_n)$ is a $q$-bounded sequence in $D(q)$ and $f_n \to f$ in $L^p(X,\mu)$, then $f \in D(q)$ and $f_n \to f$ $q$-weakly. 
% \end{lemma}
% \begin{proof}
%  Let $C = \sup_n q(f_n)$. The lower semicontinuity of $q$ implies $q(f) \leq C$ so that $f \in D(q)$. Moreover, for $\alpha \in\R \setminus\{0\}$ and $g \in D(q)$ we have
%  %
%  $$q(g) \leq \liminf_{n \to \infty}q(g - \alpha(f-f_n)) \leq q(g) + 2 \liminf_{n \to \infty}( \alpha q(g,f-f_n)) + \alpha^2 C.$$
%  %
%  Rearranging this inequality, dividing by $|\alpha|$ and taking $\alpha \to 0\pm$ yields 
%  %
%  $$0 \leq \liminf_{n \to \infty} (\pm q(g,f-f_n)).$$ 
%  %
%  This shows $f_n \to f$ $q$-weakly.
%  \end{proof}

Sometimes closedness of a quadratic form on $L^p(X,\mu)$ is equivalent to $(D(q)/\ker q,q)$ being a Hilbert space (i.e. to ensure completeness the convergence with respect to $\av{\cdot}_p$ can be omitted). The following theorem shows that this is precisely the case if  $(D(q),q)$ continuously embeds into a quotient of $L^p(X,\mu)$. It is a special case of \cite[Theorem~1.38]{Schmi3}. Since here we only deal with metrizable topological vector spaces, the proof simplifies and we include it for the convenience of the reader.   

\begin{theorem}\label{theorem:2 imply the third}
 Let $p \in \{0,2\}$ and let $q\colon L^p(X,\mu) \to [0,\infty]$ be a quadratic form. Each two of the following statements imply the third.
 \begin{enumerate}[(i)]
  \item $q$ is closed.
  \item $\ker q$ is closed in $L^p(X,\mu)$ and the embedding 
  $$\iota\colon (D(q)/\ker q, q) \to L^p(X,\mu) / \ker q,\quad f \mapsto f$$
  is continuous. Here, $L^p(X,\mu) / \ker q$ is equipped with the quotient topology.  
  \item $(D(q)/\ker q,q)$ is a Hilbert space. 
 \end{enumerate}
\begin{proof}

In both cases $p = 0, 2$, we equipped $L^p(X,\mu)$ with the translation invariant metric $d_p(f,g) = \av{f-g}_p$. If $F \subseteq L^p(X,\mu)$ is a closed subspace, then
$$d_F(f + F,g + F) = \inf \{\av{f-g + h}_p \mid h \in F\}$$
is a translation invariant metric inducing the quotient topology on $L^p(X,\mu)/F$. With respect to this metric the quotient space is complete.

 (i) \& (ii) $\Rightarrow$ (iii):  Let $(f_n + \ker q)$ be $q$-Cauchy. Since the metrics are translation invariant,  (ii) implies that $(f_n + \ker q)$ is Cauchy in  $L^p(X,\mu)/\ker q$. By completeness of the quotient we obtain $f \in L^p(X,\mu)$ such that $f_n + \ker q \to f + \ker q$ in $L^p(X,\mu)/\ker q$ and  hence $f_n + k_n \to f$ in $L^p(X,\mu)$ for some $k_n \in \ker q$. Using lower semicontinuity guaranteed by (i), we obtain 
 $$q(f - f_n) \leq \liminf_{m \to \infty} q(f_m + k_m - f_n) = \liminf_{m \to \infty} q(f_m - f_n).  $$
 This yields $f_n + \ker q \to f + \ker q$ with respect to $q$.

 (i) \& (iii) $\Rightarrow$ (ii): Lower semicontinuity implies that $\ker q$ is closed. Since  $(D(q)/\ker q,q)$ is complete and by the closedness of $\ker q$ the space $L^p(X,\mu) / \ker q$ is complete, we can use the closed graph theorem (which holds for maps between complete metrizable topological vector spaces, see e.g. \cite{Hus}). Hence, it suffices to show that the map $\iota$ is closed. To this end, consider a sequence $(f_n + \ker q)$ with $f_n + \ker q \to f + \ker q$ with respect to $q$ and $f_n + \ker q \to g + \ker q$ in $L^p(X,\mu)/\ker q$. Using lower semicontinuity guaranteed by (i) (which passes to the quotient space), we obtain 
 $$q(f-g) \leq \liminf_{n \to \infty} q(f - f_n) = 0.$$
 This yields $f + \ker q = g + \ker q$ and closedness of $\iota$ is established.

 (ii) \& (iii) $\Rightarrow$ (i): Let $(f_n)$ be $L^p(X,\mu)$-Cauchy and $q$-Cauchy. We need to show that $(f_n)$ converges with respect to the form topology. By completeness of $L^p(X,\mu)$ there exists $f \in L^p(X,\mu)$ with $f_n \to f$.  Since also $(D(q)/\ker q,q)$ is a Hilbert space, we have $q(g -f_n) \to 0$ for some $g \in D(q)$.   Then (ii) implies $f_n + k_n \to g$ in $L^p(X,\mu)$ for some $k_n \in \ker q$.  Altogether, $(k_n)$ converges  in $L^p(X,\mu)$ to $k = g-f$. Since $\ker q$ is closed and $k_n \in \ker q$, we obtain $g-f \in \ker q$. This implies $q(f-f_n) \to 0$ and we obtain $f_n \to f$ with respect to the form topology.   
\end{proof}

\end{theorem}

\subsection{Extensions of positivity preserving operators}\label{subsection:extension of positivity preserving operators}

Let  $T \colon L^p(X,\mu) \to L^p(X,\mu)$ be a positivity preserving operator, i.e., $f \geq 0$ implies $Tf \geq 0$. We abuse notation and extend it to an operator $T \colon L^+(X,\mu) \to L^+(X,\mu)$ by letting 
$$T f = \lim_{n \to \infty} T f_n$$
for a sequence $(f_n)$ in $L^p_+(X,\mu)$ with $f_n \nearrow f$ $\mu$-a.s. The limit always exists $\mu$-a.s. and   does not depend on the choice of the sequence $(f_n)$ (here we use the continuity of positivity preserving operators, see e.g. \cite[Proposition 1.3.5]{MN}). This extension is a linear map on the cone $L^+(X,\mu)$.  It follows from the definition that $T$ satisfies Fatou's lemma, i.e., for any sequence $(f_n)$ in $L^+(X,\mu)$ we have
$$T \liminf_{n \to \infty}f_n \leq \liminf_{n \to \infty} Tf_n.$$

If $f \in L^0(X,\mu)$ such that $T|f| \in L^0(X,\mu)$, then we set $Tf = Tf_+ - Tf_-$. This yields a linear operator on $L^0(X,\mu)$ with domain $\{f \in L^0(X,\mu) \mid T|f|  \in L^0(X,\mu)\}$, which extends the given operator $T$. Moreover, there is also a version of Lebesgue's dominated convergence theorem for the extension. If $g \in L_+^0(X,\mu)$ such that $Tg \in L^0_+(X,\mu)$, then $f_n \to f$ in $L^0(X,\mu)$ and $|f_n| \leq g$ imply $T|f| \in L^0(X,\mu)$ and $Tf = \lim_{n \to \infty} Tf_n$ in $L^0(X,\mu)$.
%

%

%

%

%We say that $T$ is positivity improving if $f \geq 0$ implies $Tf > 0$.

\section{The  Beurling-Deny criteria, excessive functions and extended forms} \label{section:BD}

In this section we review properties of forms satisfying the first Beurling-Deny criterion. Most of the results here are known but we include proofs for two reasons: 1)~Some of the proofs in the literature use parts of the theory we develop in subsequent sections. This can lead to intransparent cross-references, which we try to avoid. 2)~Making consequent use of lower semicontinuity and excessive functions leads to shorter proofs of some known results.

\subsection{Basics and excessive functions}

Let $p \in \{0,2\}$ and let $q \colon L^p(X,\mu) \to [0,\infty] $ be a  quadratic form. We say that $q$ satisfies the {\em first Beurling-Deny criterion} if $q(|f|) \leq q(f)$ for all $f \in L^p(X,\mu)$. A function $h \in L^0(X,\mu)$ (with $h_- \in L^2(X,\mu)$ in the case $p = 2$) is called {\em $q$-excessive} if $q(f \wedge h) \leq q(f)$ for all $f \in L^p(X,\mu)$, where $f \wedge h = \min\{f,h\}$. If the constant function $1$ is $q$-excessive, then $q$ is said to satisfy the {\em second Beurling-Deny criterion}. If, additionally, $p = 2$ and $D(q)$ is dense in $L^2(X,\mu)$, then $q$ is called {\em Dirichlet form}. As is common in the literature, below $\cE$ will denote a Dirichlet form on $L^2(X,\mu)$.

\begin{remark}
\begin{enumerate}[(a)]
 \item For closed forms the existence of a nonnegative excessive function implies the first Beurling-Deny criterion. In particular, closed forms with the second Beurling-Deny criterion satisfy the first. The question which forms satisfying the first Beurling-Deny criterion admit an excessive function is of interest and will be discussed in detail below.
 \item If $p = 2$,  we could also look at lower semibounded quadratic forms. There are two reasons why if one assumes one of the Beurling-Deny criteria, it suffices to consider nonnegative forms: 1) A lower semibounded quadratic form satisfying the second Beurling-Deny is always nonnegative. 2) A quadratic form $q$ satisfies the first Beurling-Deny criterion if and only if for $\alpha \in \R$ the quadratic form $q + \alpha \av{\cdot}^2_2$ satisfies the first Beurling-Deny criterion. Hence, $q$ is either nonnegative anyway or one can add a multiple of the square of the $L^2$-norm to obtain a nonnegative form without loosing Beurling-Deny criteria.   
\end{enumerate}
\end{remark} 
\begin{lemma}\label{lemma:form domain lattice}
 Let $p \in \{0,2\}$ and let $q \colon L^p(X,\mu) \to [0,\infty] $  be a quadratic form satisfying the first Beurling-Deny criterion.  For all $f,g \in L^p(X,\mu)$ 
 $$q(f \wedge g) + q(f \vee g) \leq q(f) + q(g).$$
In particular, $D(q)$ is a lattice and functions in $\ker q$ are $q$-excessive.
 \end{lemma}
\begin{proof}
 The parallelogram identity  and the first Beurling-Deny criterion yield 
 \begin{align*}
  2q(f \wedge g) + 2q(f \vee g) &= q(f \wedge g + f \vee g) + q(f \wedge g - f \vee g)\\
  &= q(f + g) + q(|f-g|)\\
  &\leq q(f + g) + q(f-g)\\
  &= 2q(f) + 2q(g). \hfill \qedhere
 \end{align*}
\end{proof}
The following lemma shows that the set of nonnegative excessive functions is closed under local convergence in measure.
\begin{lemma}\label{lemma:limit of excessive functions is excessive}
 Let $p \in \{0,2\}$ and let  $q\colon L^p(X,\mu) \to [0,\infty]$ be a closed quadratic form. Let $(h_n)$ be a sequence of nonnegative $q$-excessive functions converging in $L^0(X,\mu)$ to $h$. Then $h$ is $q$-excessive. 
\end{lemma}
\begin{proof}
 Since $h_n$ and $h$ are nonnegative, we have $f\wedge h_n = f_+ \wedge h_n - f_- \to f_+ \wedge h - f_- = f\wedge h$ in $L^p(X,\mu)$ (this is clear for $p = 0$, for $p = 2$ it follows from Lebesgue's dominated convergence theorem). Hence, lower semicontinuity of $q$ on $L^p(X,\mu)$ implies
 $$q(f \wedge h) \leq \liminf_{n \to \infty} q(f\wedge h_n) \leq q(f).$$
 This shows that $h$ is $q$-excessive.
\end{proof}

The following lemma is due to Ancona, see \cite[Proposition~4]{Anc}. It shows that for forms satisfying the first Beurling-Deny criterion taking the absolute value is continuous with respect to the form topology (which is induced by the metric $d_q(f,g) = q(f-g)^{1/2} + \av{f-g}_p$ on $D(q)$). The setting of \cite{Anc} is a bit different  and therefore we explain why Ancona's proof can be applied in our situation. 
 
 \begin{lemma}\label{lemma:ancona}
 For $p \in \{0,2\}$ let $q$ be a closed quadratic form on $L^p(X,\mu)$ satisfying the first Beurling-Deny criterion. Then  
 $$\lvert\cdot \rvert \colon D(q) \to D(q), \quad f \mapsto |f|$$
 is continuous with respect to the form topology.
 \end{lemma}
\begin{proof}
 The equations necessary for proving that $f_n \to f$ with respect to the form topology implies $|f_n| \to |f|$ with respect to $q$ are contained in the proof of \cite[Proposition~4]{Anc}. In contrast to our assumptions, the setting of \cite{Anc} requires that $(D(q),q)$ is a Hilbert space and that $(D(q),q)$ continuously embeds into $L^0(X,\mu)$. However, what Ancona really uses in the proof of \cite[Proposition~4]{Anc} is that if $(f_n)$ in $D(q)$ is $q$-bounded and $f_n \to f$ in $L^p(X,\mu)$, then $f \in D(q)$ and $f_n \to f$ $q$-weakly.  This is a standard result if $p = 2$ and proved in \cite[Lemma~A.5]{Schmi2} for $p = 0$.
%  
%  
%  
%  Suppose $f_n \to f$ in  the form topology. It suffices to show $(f_{n})_+ \to f_+$ in the form topology.  Clearly, $(f_{n})_+ \to f_+$ in $L^p(X,\mu)$ and since $q$ satisfies the first Beurling-Deny criterion, $(f_{n})_+$ is $q$-bounded. We obtain   $(f_{n})_+ \to f_+$ $q$-weakly by Lemma~\ref{lemma:weak convergence}. 
%  %
%  For  strong convergence we compute 
%  %
%  \begin{align*}
%   q((f_{n})_+ - f_+) &= -q((f_n)_+ - f_+, f_+) + q((f_n)_+ - f_+, (f_n)_+)\\
%   &= q(f_+ - (f_n)_+, f_+) + q(f -  f_n, (f_n)_+) + q(f_- - (f_n)_-, (f_n)_+).
%  \end{align*}
%  %
%  By what we have shown so far the first and the second term on the right side tend to $0$.  For the third term  we first note that the convergence $f_n \to f$ with respect to $q$ and lower semicontinuity of $q$ yield
%  %
%  $$q(f_+, f_-) = \frac{1}{4}(q(|f|) - q(f)) \leq \liminf_{n \to \infty} \frac{1}{4}((q(|f_n|) - q(f_n)) =  \liminf_{n \to \infty}q((f_n)_+,(f_n)_-). $$
%  %
%  Using this and $(f_n)_+ \to  f_+$ $q$-weakly, we obtain 
%  %
%  \begin{align*}
%   \limsup_{n \to \infty} q(f_- - (f_n)_-, (f_n)_+)&= q(f_-, f_+) - \liminf_{n \to \infty}  q((f_n)_-, (f_n)_+) \leq 0. 
%  \end{align*}
%  %
%   To treat the third term we compute 
%  %
% \begin{align*}
%  q(f_- - (f_n)_-, (f_n)_+) &= q(f_- - (f_n)_-, (f_n)_+)
% \end{align*}
%  %
%  
%  
%  
%  
%  Then $f_n \wedge f \to f$ in $L^p(X,\mu)$ and 
%  %
%  $$q(f- f\wedge f_n) = q((f-f_n)_+) \leq q(f - f_n)$$
%  %
%  shows $f \wedge f_n \to f$ with respect to the form topology. We obtain further 
%  
\end{proof}
% 
% 
% \begin{remark}
%  In the proof we used that for $q$-bounded $(f_n)$ the convergence $f_n \to f$ in $L^p(X,\mu)$ yields $f_n \to f$ $q$-weakly. This is a well known fact if $p = 2$, but the standard proof relying on the Banach-Saks theorem does not work in the case $p = 0$. 
% \end{remark}

\begin{assumption}
 From now on we assume that all quadratic forms on $L^p(X,\mu)$, $p \in \{0,2\}$, are {\em densely defined}, i.e., $D(q)$ is dense in $L^p(X,\mu)$.
\end{assumption}

\begin{remark}
 For forms $q$ on $L^p(X,\mu)$, $p \in \{0,2\}$, satisfying the first Beurling-Deny criterion, this assumption is not a restriction. In this case, $D(q)$ is a lattice  and hence the $L^p(X,\mu)$-closure of $D(q)$ is a closed vector lattice in $L^p(X,\mu)$. Such lattices are given by $L^p(Y,\mathfrak B,\mu|_{\mathfrak B})$,  where $Y\subseteq X$ and $\mathfrak B$ is a $\sigma$-algebra on $Y$ contained in the original $\sigma$-algebra. Hence, $q$ can always be considered to be densely defined on some $L^p$-space.
\end{remark}

Let $q$ be a closed form on $L^2(X,\mu)$. For $\alpha > 0$ we define the quadratic form 
$$q_\alpha \colon L^2(X,\mu) \to [0,\infty], \quad q_\alpha(f) = q(f) + \alpha \av{f}_2^2.$$
By the Riesz representation theorem for $\alpha > 0$ and $f \in L^2(X,\mu)$ there exists a unique element $G_\alpha f \in D(q)$ such that 
$$q(G_\alpha f,g) + \alpha \as{G_\alpha f,g} = \as{f,g}$$
for all $g \in D(q)$. The induced family of self-adjoint operators $(G_\alpha)_{\alpha > 0}$ on $L^2(X,\mu)$ is called {\em resolvent family} of $q$. Since $q$ is densely defined, it corresponds to a unique nonnegative self-adjoint operator $L$, the so-called {\em generator of $q$} through $G_\alpha = (L+\alpha)^{-1}$, $\alpha > 0$. Moreover, it is strongly continuous, i.e., for $f \in L^2(X,\mu)$ we have $\alpha G_\alpha f \to f$, as $\alpha  \to \infty$. Using spectral calculus of $L$, we obtain a self-adjoint $C_0$-semigroup $(T_t)_{t > 0} = (e^{-tL})_{t >0}$ on $L^2(X,\mu)$, which we call {\em semigroup associated with $q$}.

It is well-known that $G_\alpha$, $\alpha > 0,$ and $T_t, t>0,$ are positivity preserving if and only if $q$ satisfies the first Beurling-Deny criterion. Moreover, $(G_\alpha)_{\alpha > 0}$ and $(T_t)_{t > 0}$ are Markovian (i.e. $0 \leq f \leq 1$ implies $0 \leq \alpha G_\alpha f \leq 1$ respectively $0 \leq T_t f\leq 1$) if and only if $q$ satisfies the second Beurling-Deny criterion. See e.g. \cite[Theorem XIII.50 and Theorem XIII.51]{RSIV}. We need the following extension of this result, which is basically taken from \cite[Proposition~4]{Kaj}.

\begin{lemma}\label{lemma:characterization of excessive functions}
 Let $q\colon L^2(X,\mu) \to [0,\infty]$ be a quadratic form satisfying the first Beurling-Deny criterion.  For   $h \in L_+^0(X,\mu)$ the following assertions are equivalent. 
 \begin{enumerate}[(i)]
  \item $h$ is $q$-excessive.
  \item  $\alpha G_\alpha h \leq h$ for all $\alpha > 0$.
  \item $T_t h \leq h$ for all $t > 0$.
 \end{enumerate}
If, additionally, $h \in D(q)$, this is equivalent to $q(h,g) \geq 0$ for all nonnegative $g \in D(q). $
\end{lemma}
\begin{proof}
This follows directly from Ouhabaz' invariance criterion for closed convex sets under the resolvent respectively semigroup, see \cite{Ouh}. See also the proof of \cite[Proposition~4]{Kaj} for more details.
\end{proof}
\begin{remark}
There are two immediate consequences of this lemma, which we will use below:
\begin{enumerate}[(a)]
 \item Assertion (ii)  yields that if $ h \in L^0_+(X,\mu)$ is excessive and strictly positive, then $\alpha G_\alpha$ and $T_t$ extend to  contractions on $L^\infty_h(X,\mu)$. By duality we obtain that $\alpha G_\alpha$ and $T_t$ extend to a contraction on $L^1(X,h\mu)$.
 \item The characterization of excessive functions in the form domain shows that for nonnegative $f \in L^2(X,\mu)$ and $\beta \geq \alpha$, the resolvent $G_\alpha f$ is $q_\beta$-excessive.
\end{enumerate}
\end{remark}

If $q$ satisfies the first Beurling-Deny criterion,  for $\alpha \leq \beta$ the resolvent identity $G_\alpha   = G_\beta + (\beta - \alpha)G_\alpha G_\beta$ extends from $L^2_+(X,\mu)$ to the extended operators on $L^+(X,\mu)$. This implies that for $f \in L^+(X,\mu)$ the map $ (0,\infty) \to L^+(X,\mu),\, \alpha \mapsto G_\alpha f$ is monotone decreasing.  Hence, for each nonnegative $f \in L^+(X,\mu)$ the limit
$$Gf =  \lim_{\alpha \to 0+} G_\alpha f$$
exists in $L^+(X,\mu)$. The map $G \colon L^+(X,\mu) \to L^+(X,\mu)$ is called the {\em Green operator} of $q$.

\begin{lemma}\label{lemma:existence of resolvents}
Let  $q \colon L^2(X,\mu) \to [0,\infty]$ be a quadratic form satisfying the first Beurling-Deny criterion. 
\begin{enumerate}[(a)]
 \item Let $g \in L^+(X,\mu)$ and suppose that $Gg \in L^0(X,\mu)$.  Then $Gg$ is $q$-excessive.
 \item Suppose  $g \in L_+^2(X,\mu)$ is strictly positive such that for all $f \in D(q)$
 $$\int_X  |f| gd\mu \leq q(f)^{1/2}.$$
 Then $Gg$ is strictly positive and $Gg \in L_+^1(X,g\mu)$. 
\end{enumerate}
\end{lemma}
\begin{proof}
 (a): According to the previous lemma, for $f \in L^2_+(X,\mu)$ and $\beta \leq \alpha$ the resolvent   $G_\beta f$ is $q_\alpha$-excessive. Now $Gg$ is the limit in $L^0(X,\mu)$ of functions of the form $G_\beta f$ with $\beta \leq \alpha$ and $f \leq g$. Hence, $Gg$ is also $q_\alpha$-excessive by Lemma~\ref{lemma:limit of excessive functions is excessive}. We obtain 
 $$q(f \wedge Gg) + \alpha \av{f \wedge Gg}_2^2 = q_\alpha(f \wedge Gg) \leq q_\alpha(f) =  q(f) + \alpha \av{f}_2^2.$$
 Letting $\alpha \to 0+$ yields the claim. 

%  
%  Since $G_\alpha g$ and $Gg$ are nonnegative, for $f \in L^2(X,\mu)$ we have $f \wedge Gg \in L^2(X,\mu)$ and  $f \wedge G_\alpha g \to f \wedge Gg$ in $L^2(X,\mu)$. The lower semicontinuity of $q_\alpha$ yields
%  %
%  $$q_\alpha(f \wedge Gg) \leq \liminf_{\beta \to 0+} q_\alpha(f \wedge G_\beta g) \leq q_\alpha(f).$$
%  %
%  Here we used that $G_\beta f$ is $q_\alpha$-excessive for $\beta \leq \alpha$.  Letting $\alpha \to 0+$ and using $f \wedge Gg, f \in L^2(X,\mu)$ shows that claim.
%  
 (b): Strict positivity follows from   $\alpha G_\alpha g \to g$, as $\alpha \to \infty$, and the fact that $\alpha \to G_\alpha g$ is decreasing.  Using the definition of $G_\alpha$ we obtain 
 $$\left(\int_X g G_\alpha g d\mu \right)^2 \leq q(G_\alpha g) = \as{g,G_\alpha g} - \alpha \av{G_\alpha g}_2^2. $$
 This implies $\int_X g G_\alpha g d\mu \leq 1$ and we obtain the statement from the monotone convergence theorem.  
\end{proof}

\subsection{The extensions $q_e$ and $q^+$}

Every quadratic form $q$ on $L^2(X,\mu)$ can be interpreted to be a quadratic form on $L^0(X,\mu)$ by letting $q (f) = \infty$ for $f \in L^0(X,\mu) \setminus L^2(X,\mu)$. The following lemma shows that forms on $L^2(X,\mu)$ satisfying the first Beurling-Deny criterion are closable when viewed in this sense as forms on $L^0(X,\mu)$.

\begin{proposition}[Existence of the extended form]\label{prop:existence of extended form}
 Let $q$ be a  closed quadratic form on $L^2(X,\mu)$ satisfying the first Beurling-Deny criterion. Then $q$ is closable on $L^0(X,\mu)$ and its closure $q_{e}$ is a quadratic form satisfying the first Beurling-Deny criterion.  Moreover, $D(q_e)  \cap L^2(X,\mu) = D(q)$ and a function $h \in L_+^0(X,\mu)$ is $q$-excessive if and only if it is $q_e$-excessive.  
\end{proposition}
\begin{proof}
 According to Lemma~\ref{lemma:closability} we need to show that $q$ is lower semicontinuous with respect to $L^0(X,\mu)$-convergence on its domain. Thus, let $(f_n)$ and $f \in D(q)$ such that $f_n \to f$ in $L^0(X,\mu)$.  Without loss of generality we can assume 
 $\liminf_{n \to \infty}q(f_n)  < \infty$
 and $f_n \to f$ $\mu$-a.e. (else choose a subsequence). 
 
 {\em Case 1:} There exists a $q$-excessive function $h > 0$.  
 
 First we assume there exists $C > 0$ such that $|f_n| \leq C h$ for all $n \in \N$.  For $\varphi \in L^1(X,h\mu) \cap L^2(X,\mu)$ and $\alpha > 0$ we obtain 
 $$q(f,G_\alpha \varphi) = \as{f, \varphi - \alpha G_\alpha \varphi} = \lim_{n \to \infty} \as{f_n, \varphi - \alpha G_\alpha \varphi} = q(f_n,G_\alpha \varphi).$$
 For the second equality we used Lebesgue's dominated convergence theorem, which can be applied since $\alpha G_\alpha \varphi \in L^1(X,h\mu)$ by $h$ being excessive. Elements of the form $G_\alpha \varphi$ as above are dense in $D(q)$ with respect to $q$. Since $(f_n)$ is  $q$-bounded, this implies $f_n \to f$ weakly with respect to $q$. Using the Cauchy-Schwarz inequality,  weak convergence yields 
 $$q(f) \leq \liminf_{n\to \infty} q(f_n).$$
 This shows lower semicontinuity for sequences, which are uniformly bounded by an excessive function.
 
 With this at hand we can treat the general case. Since $h$ is excessive, for any $C > 0$ the function  $Ch$ is excessive, and since $h > 0$, we have $f = \lim_{C \to \infty} (f \wedge (Ch)) \vee (-Ch)$ in $L^2(X,\mu)$. The $L^2$-lower semicontinuity of $q$, $Ch$ being excessive and the lower semicontinuity for sequences which are uniformly bounded by an excessive function yield 
 \begin{align*}
 q(f) &\leq \liminf_{C \to \infty} q((f \wedge (Ch)) \vee (-Ch))\\
 &\leq  \liminf_{C \to \infty}\liminf_{n \to \infty} q((f_n \wedge (Ch)) \vee (-Ch)) \\
 &\leq \liminf_{n \to \infty} q(f_n).
 \end{align*}
This is the claim we wanted to prove.
 
 {\em Case 2:} $q$ does not have a strictly positive excessive function.  By pointwise convergence the pointwise supremum $F = \sup_{n} |f_n|$ exists and is finite $\mu$-a.e. We let $g = \varphi/(F \vee 1)$ for some $0< \varphi \in L^1(X,\mu) \cap L^\infty(X,\mu)$. Then by construction $(f_n)$ is bounded in $L^2(X,g\mu)$.
 
 Next, for $\varepsilon > 0$ we consider the quadratic form $\tilde q \colon L^2(X,\mu) \to [0,\infty]$ defined by
 $$\tilde q (f) = q(f) + \varepsilon \int_X f^2 g d\mu. $$
 By Lemma~\ref{lemma:existence of resolvents} the Green operator of the form $\tilde q$ applied to $g$ exists and is strictly positive. Therefore, $\tilde q$ has an  excessive function $h > 0$. Moreover,  $\liminf_{n \to \infty}\tilde q(f_n) < \infty$ by our choice of $g$. Hence, Case~1 shows 
 $$\tilde q(f) \leq \liminf_{n\to \infty}\tilde q(f_n).$$
 Since $(f_n)$ is bounded in $L^2(X,g\mu)$, we can take the limit $\varepsilon \to 0+$ and obtain the desired lower semicontinuity.
 
 $q_e$ satisfies the first Beurling Deny criterion: For $f \in D(q_e)$ we can choose $f_n \in D(q)$ with $f_n \to f$ in $L^0(X,\mu)$ and $q(f_n) \to q_e(f)$. The lower semicontinuity of $q_e$ and $q_e = q$ on $D(q)$ yield
 $$q_e(|f|) \leq \liminf_{n \to \infty}q_e(|f_n|)= \liminf_{n \to \infty}q(|f_n|) \leq  \lim_{n \to \infty}q(f_n) = q_e(f). $$

 It remains to prove $D(q_e) \cap L^2(X,\mu) = D(q)$ and the statement on excessive functions. The inclusion   $D(q) \subseteq D(q_e) \cap L^2(X,\mu)$ is trivial. Let $f \in D(q_e) \cap L^2(X,\mu)$ and let $(f_n)$ be a $q$-Cauchy sequence with  $f_n \to f$ in $L^0(X,\mu)$. Without loss of generality we can assume $f_n \to f$ $\mu$-a.e. (else choose a subsequence).   Lebesgue's dominated convergence theorem implies 
 $$f = \lim_{n \to \infty} (f \wedge |f_n|) \vee (-|f_n|) = \lim_{n \to \infty} \lim_{m \to \infty} (f_m \wedge |f_n|) \vee (-|f_n|)$$ 
 in $L^2(X,\mu)$. The $L^2$-lower semicontinuity of $q$ and the subadditivity of $q$ with respect to taking suprema yield
 $$q(f) \leq \liminf_{n \to \infty}\liminf_{m \to \infty} q((f_m \wedge |f_n|) \vee (-|f_n|)) \leq 3 \lim_{n\to \infty}q(f_n) < \infty.$$
 This shows $f \in D(q)$.
 
 Since $q = q_e$ on $L^2(X,\mu)$ (here we use $D(q_e) \cap L^2(X,\mu) = D(q)$), nonnegative $q_e$-excessive functions are  $q$-excessive. Now suppose $h \in L^0_+(X,\mu)$ is $q$-excessive. For $f \in D(q_e)$ we choose $(f_n)$ in $D(q)$ with $f_n \to f$ in $L^0(X,\mu)$ and $q(f_n) \to q_e(f)$. The lower semicontinuity of $q_e$ and $q = q_e$ on $D(q)$ yield
 $$q_e(f \wedge h) \leq \liminf_{n \to \infty} q_e(f_n \wedge h)  = \liminf_{n \to \infty} q(f_n \wedge h) \leq \liminf_{n \to \infty} q(f_n) = q_e(f).$$
 This shows that $h$ is also $q_e$-excessive.
 \end{proof}

 \begin{definition}[Extended form]
 The closed quadratic form  $q_e$ on $L^0(X,\mu)$ introduced in the previous proposition is called the {\em extended form of $q$}. 
 \item Note that it follows directly from $D(q) = D(q_e) \cap L^2(X,\mu)$ and the lower semicontinuity of $q_e$ that $q$ and $q_e$ have the same excessive functions.
 \end{definition}

\begin{remark}
\begin{enumerate}[(a)]
 \item Using the formula for $q_e$ from Lemma~\ref{lemma:closability} and the characterization of convergence in $L^0(X,\mu)$ in terms of $\mu$-a.e. convergent subsequences, it is easy to see that if $\cE$ is a Dirichlet form, then $D(\cE_e)$ is the extended Dirichlet space of $\cE$ and $\cE_e$ is the extension of $\cE$ to the extended Dirichlet space, see e.g. \cite[Definition~1.1.4]{CF} for the definition of the extended Dirichlet space. 
 \item The first assertion of this proposition (namely closability of $q$ on $L^0(X,\mu)$) is  known.  For Dirichlet forms it is equivalent to the fact that the extension of the Dirichlet form to the extended Dirichlet space is well-defined, see \cite[Theorem~1.5.2]{FOT},  which has its origin in \cite{Sil}. In the generality we use here, the lower semicontinuity of $q$ on its domain with respect to a.e.-convergence is the main result of \cite{Schmu2}. As mentioned in the introduction, the proof given in \cite{Schmu2} uses parts of the criticality theory we develop below. Therefore, we gave an independent proof.  The idea of first proving lower semicontinuity for forms with excessive functions and then extending it to the general case is taken from \cite{Schmu2} but our proof of Case~1, which is based on properties of excessive functions, is much shorter than the one in \cite{Schmu2}.  For Dirichlet forms a version of the argument is contained in the proof of \cite[Theorem~1.59]{Schmi3}. 
  
 \item To the best of our knowledge the literature only contains that the extension of $q$ to the extended space $D(q_e)$ is lower semicontinuous on its domain, i.e., $f_n \to f$ in $L^0(X,\mu)$ and $f \in D(q_e)$ implies $q_e(f) \leq \liminf_{n \to \infty}q_e(f_n)$, see e.g. \cite[Corollary~1.1.9]{CF}. The statement of our proposition is stronger. Since our approach automatically gives closedness of $q_e$, we have lower semicontinuity of $q_e$ on the whole space $L^0(X,\mu)$. This is not a mere technicality, but crucial for our considerations below.
 \end{enumerate}
 \end{remark}

\begin{lemma}\label{lemma:extension of resolvent to extended space}
 Let $q$ be a closed quadratic form on $L^2(X,\mu)$ satisfying the first Beurling-Deny criterion. For $\alpha > 0$ the operator $\alpha G_\alpha \colon D(q) \to D(q)$ extends by continuity with respect to $q$  to a contraction  $\alpha G_\alpha \colon D(q_e) \to D(q_e)$ with respect to $q_e$. This extension is compatible with the extension of $\alpha G_\alpha$ to $L^0(X,\mu)$, which was discussed in Subsection~\ref{subsection:extension of positivity preserving operators}. Moreover, for any $f \in D(q_e)$ we have $f - \alpha G_\alpha f \in L^2(X,\mu)$ and
 $$\lim_{\alpha \to 0+}q_e(\alpha G_\alpha f) = 0.$$
\end{lemma} 
 \begin{proof}
 By the spectral theorem for the (nonnegative) generator of $q$ we have
 $$q(\alpha G_\alpha f)  = \int_{[0,\infty)} \frac{\lambda \alpha^2}{(\lambda + \alpha)^2} d\sigma_f(\lambda) \leq \int_{[0,\infty)} \lambda d\sigma_f(\lambda) = q(f) $$
 and 
 $$\alpha \av{f - \alpha G_\alpha f}^2 = \int_{[0,\infty)} \frac{\alpha \lambda^2}{(\lambda + \alpha)^2} d\sigma_f (\lambda) \leq \int_{[0,\infty)} \lambda d\sigma_f(\lambda) = q(f),$$
 where $\sigma_f$ is the spectral measure of the generator at $f \in L^2(X,\mu)$. Using Lebesgue's dominated convergence theorem, this shows $\lim_{\alpha \to 0+}q(\alpha G_\alpha f) = 0$ for $f \in D(q)$.
 
Now let $f \in D(q_e)$ and let $(f_n)$ be $q$-Cauchy with $f_n \to f$ in $L^0(X,\mu)$. By the lower semicontinuity of $q_e$ we have $f_n \to f$ in the $q_e$-form topology.  The inequalities above show  that  $(\alpha G_\alpha f_n)$ is $q$-Cauchy and that $(f_n - \alpha G_\alpha f_n)$ is $L^2(X,\mu)$-Cauchy. Hence, there exists $R_\alpha f \in L^0(X,\mu)$ such that  $\alpha G_\alpha f_n \to  R_\alpha f$ in $L^0(X,\mu)$ with $f - R_\alpha f \in L^2(X,\mu)$. This implies $R_\alpha f \in D(q_e)$ 
 and 
 $$q_e(R_\alpha f)\leq \liminf_{n \to \infty}q(\alpha G_\alpha f_n) \leq \liminf_{n \to \infty}q(f_n) = q_e(f).$$
 We proved that $\alpha G_\alpha \colon D(q)\to D(q)$ extends to a contraction $R_\alpha \colon D(q_e) \to D(q_e)$ with the continuity property that $(f_n)$ in $D(q)$ with $f_n \to f$ in the form topology yields $\alpha G_\alpha f_n \to R_\alpha f$ in $L^0(X,\mu)$.

 It remains to show $ R_\alpha f = \alpha G_\alpha f$, where the right side of this equation denotes the extension of $\alpha G_\alpha$ discussed in Subsection~\ref{subsection:extension of positivity preserving operators}. Lemma~\ref{lemma:ancona} yields  $|f_n| \to |f|$ in the $q_e$-form topology. Fatou's lemma for positivity preserving operators  and our definition of $R_\alpha$ show
 $$\alpha G_\alpha |f| \leq \lim_{n \to \infty} \alpha G_\alpha |f_n| = R_\alpha |f| \in D(q_e). $$
 This implies that $\alpha G_\alpha f$  exists. Now consider $g_n = (f_n\wedge |f|) \vee (-|f|)$. Since $D(q_e) \cap L^2(X,\mu) = D(q)$, we have $g_n \in D(q)$. Moreover, by Lemma~\ref{lemma:ancona} $(g_n)$ converges to $f$ in the form topology. Using $|g_n| \leq f$ we can apply Lebesgue's theorem for positivity preserving operators and obtain 
 $$\alpha G_\alpha f = \lim_{n \to \infty} \alpha G_\alpha g_n = R_\alpha f.$$
For the last equality we used our definition of $R_\alpha$. 

The convergence $\lim_{\alpha \to 0+}q_e(\alpha G_\alpha f) = 0$ follows from the fact that $D(q)$ is dense in $D(q_e)$ with respect to the form topology and that $\alpha G_\alpha$ is a contraction.
 \end{proof}
 
 \begin{remark}
   The previous lemma is a resolvent version of \cite[Lemma~2.8]{TU}, which treats the semigroup. We included a short proof because in \cite{TU} the relation of the extension of $G_\alpha$ to $D(q_e)$ by continuity and by positivity are not explained. 
 \end{remark}

There is another lower semicontinuous extension of $q_e$ to $L^+(X,\mu)$, which we discuss next.   Since the topologies of $L^+(X,\mu)$ and $L^0(X,\mu)$ agree on $L^0_+(X,\mu)$, the restriction of $q_e$ to $D(q_e) \cap L^+(X,\mu)$  is lower semicontinuous with respect to $L^+(X,\mu)$-convergence. This implies that $q^+ \colon L^+(X,\mu) \to [0,\infty]$ given by
$$ \quad  q^+(f) = \begin{cases}
                  \lim\limits_{n \to \infty} q_e(f_n)  &\text{if there ex. $q_e$-Cauchy sequence } (f_n) \text{ with } f_n \to f \text{ in }L^+(X,\mu)\\
                  \infty &\text{else}
                 \end{cases}$$
is well-defined. The same arguments used in the  proof showing that $q_e$ is lower semicontinuous also yield that $q^+$ is lower semicontinuous, see the proof of \cite[Lemma~A.3]{Schmi2}. The functional $q^+ \colon L^+(X,\mu) \to [0,\infty]$ is not a quadratic form as $L^+(X,\mu)$ is not even a vector space. However, it is readily verified that it is homogeneous, i.e., $\lambda^2 q^+(f) = q^+(\lambda f)$ for all $\lambda \geq 0$ and $f \in L^+(X,\mu)$. 

The lower semicontinuity of $q^+$ and the corresponding inequality for $q_e$, Lemma~\ref{lemma:form domain lattice},  show
$$q^+(f\wedge g) + q^+(f\vee g) \leq q^+(f) + q^+(g).$$

With this at hand, if we let $D(q^+) = \{f \in L^+(X,\mu) \mid q^+(f) < \infty\}$, the same proof as the one we gave for the identity $D(q_e) \cap L^2(X,\mu) = D(q)$ yields 
$$D(q^+) \cap L_+^0(X,\mu)  = D(q_e) \cap L^0_+(X,\mu).$$ 
Moreover, if $h$ is $q$-excessive (or equivalently $q_e$-excessive), then also 
$$q^+(f \wedge h) \leq q^+(f), \quad f \in L^+(X,\mu).$$

The functional $q^+$ can assign finite values to functions taking the value $\infty$ on a set of positive measure. However, the following lemma shows that this is only the case if $\ker q^+$ is nontrivial. 

\begin{lemma}\label{lemma:kernel q^+}
 Let $q$ be a closed quadratic form on $L^2(X,\mu)$ satisfying the first Beurling-Deny criterion. For  $f \in L^+(X,\mu)$ with $q^+(f) < \infty$ 
 %
 %the set $\{f = \infty\}$ is $q_e$-invariant and
 we have $\infty \cdot 1_{\{f = \infty\}} \in \ker q^+$. 
 %
 %In particular, if $q$ is irreducible, then 
 %
% $$D(q^+) \subseteq (D(q_e) \cap L^0_+(X,\mu)) \cup \{\infty \cdot 1_X\}.$$
 %
\end{lemma}
\begin{proof}
 Let $h = \infty \cdot 1_{\{f = \infty\}} = \lim_{\lambda \to 0+}  \lambda f. $ The lower semicontinuity and homogeneity of $q^+$ yield
 \begin{align*}
  q^+(h) \leq \liminf_{\lambda \to 0+} q^+(\lambda f) = \liminf_{\lambda \to 0+} \lambda^2 q^+(f) = 0. 
 \end{align*}
 %
 %For $f \in L^0(X,\mu)$...
\end{proof}

\begin{lemma}\label{lemma:continuity q^+}
 Let $q$ be a closed quadratic form satisfying the first Beurling-Deny criterion. Let $(f_n)$ be a sequence in $L^+(X,\mu)$ with $\sum_{k = 1}^\infty q^+(f_k) < \infty$. Then
 $$q^+(\liminf_{n \to \infty} f_n) = q^+(\limsup_{n \to \infty} f_n) = 0.$$
In particular, if $\ker q^+ = \{0\}$, then $f_n \to 0$ $\mu$-a.e.
 \end{lemma}
\begin{proof}
We only treat $\limsup f_n$, the  statement on $\liminf f_n$ follows along the same lines. We have
 $$\limsup_{n \to \infty} f_{n} = \lim_{n \to \infty}  \sup_{l \geq n} f_{l} = \lim_{n \to \infty} \lim_{N \to \infty}  \sup_{ N\geq l \geq n} f_{l}  $$
 pointwise $\mu$-a.e. and hence in $L^+(X,\mu)$. Since $q^+$ is lower semicontinuous on $L^+(X,\mu)$, we obtain 
\begin{align*}
 q^+(\limsup_{n \to \infty} f_{n}) \leq \liminf_{n \to \infty} \liminf_{N \to \infty} q^+(\sup_{ N\geq l \geq n} f_{l}) 
 \leq  \liminf_{n \to \infty} \liminf_{N \to \infty} \sum_{l = n}^N q^+ (f_l) = 0.
\end{align*}
For the second inequality we used the subadditivity of $q^+$ with respect to taking suprema. The 'in particular'-statement follows immediately. 
\end{proof}
 \begin{remark}
The previous lemma is a continuity property for the functional $q^+$. If $\ker q^+ = \{0\}$, it yields that $(D(q_e),q_e)$ continuously embeds into $L^0(X,\mu)$. This observation is exploited further in Section~\ref{section:criticality}.
 \end{remark}

 \subsection{Invariant sets and irreducibilty}

Let $q$ be a quadratic form on $L^p(X,\mu)$, $p \in \{0,2\}$. We say that a measurable set $A\subseteq X$ is {\em $q$-invariant} if 
$q(1_A f) \leq q(f)$
for all $f \in L^p(X,\mu)$. Indeed, this definition  of irreducibilty equals the usual one, see e.g. \cite[Lemma~2.32]{Schmi3}. The quadratic form $q$ is called {\em irreducible} or {\em ergodic} if every $q$-invariant set $A$ satisfies $\mu(A) = 0$ or $\mu(X\setminus A) = 0$.  Using lower semicontinuity it is easy to see that for a closed form satisfying the first Beurling-Deny criterion a set $A$ is $q$-invariant if and only if it is $q_e$-invariant. In particular,   $q$ is irreducible if and only if $q_e$ is irreducible.

\begin{proposition}
 For $p \in \{0,2\}$ let $q$ be a closed quadratic form on $L^p(X,\mu)$ satisfying the first Beurling-Deny criterion. Let $h \in L^0(X,\mu)$ be $q$-excessive. 
 \begin{enumerate}[(a)]
  \item $h_- \in \ker q$. In particular, if $\ker q = \{0\}$, then every excessive function is nonnegative.
  \item If $h$ is nonnegative, then $\{h > 0\}$ is $q$-invariant. In particular, if $q$ is irreducible, every nontrivial nonnegative excessive function is strictly positive. 
 \end{enumerate}
\end{proposition}
\begin{proof}
 (a): $h_- = (-h) \vee 0 = - (h \wedge 0).$ We obtain 
 $$q(h_-) = q(-h_-) = q(h \wedge 0) \leq q(0) = 0.$$
 This shows $h_- \in \ker q$.
 
 (b): For $f \in L^p(X,\mu)$ the identity 
 $$1_{\{h > 0\}} f = \lim_{n \to \infty} (f \wedge (nh)) \vee (-nh)$$
 holds in $L^p(X,\mu)$. Since $nh$ is also excessive, we obtain using lower semicontinuity
 $$q(1_{\{h > 0\}} f) \leq \liminf_{n \to \infty} q((f \wedge (nh)) \vee (-nh)) \leq q(f).$$
 This shows the invariance of $\{h > 0\}$.
\end{proof}
\begin{corollary}\label{coro:perron frobenius}
 For $p \in \{0,2\}$ let $q$ be an irreducible closed quadratic form on $L^p(X,\mu)$ satisfying the first Beurling-Deny criterion. Then $\ker q$ is at most one-dimensional and if $\ker q \neq \{0\}$, there exists a strictly positive $h \in L^p(X,\mu)$ such that $\ker q = \R h$.
\end{corollary}
\begin{proof}
 Assume there exists $0 \neq h \in \ker q$ (else there is nothing to show). Since $q$ satisfies the first Beurling-Deny criterion, we can assume $h \geq 0$ (else consider $|h|$ instead of $h$). Part (b) of the previous lemma yields $h > 0$ (use that functions in $\ker q$ are excessive). 
 
 Now let $h' \in \ker q_e$ and for $\alpha \in \R$ consider the function $g_\alpha =  h' - \alpha h \in \ker q$. Using part (b) of the previous lemma again, $g_\alpha$ has fixed sign, i.e., either $g_\alpha = 0$, $g_\alpha > 0$ or $g_\alpha < 0$.  Let $\alpha_0 = \sup \{\alpha \mid g_\alpha > 0\}$. Then, obviously, $0 = g_{\alpha_0} = h' - \alpha_0 h$. 
\end{proof}

\section{(Very) weak and abstract Poincaré and Hardy inequalities}

In this section we discuss very weak Poincaré and Hardy inequalities, which are valid for most quadratic forms on Hilbert spaces.  We then show how they yield weak Hardy inequalities for forms satisfying the first Beurling-Deny criterion with trivial kernel. Moreover, we prove abstract Hardy inequalites, which hold for all forms satisfying the first Beurling-Deny criterion.

Our very weak Poincaré and Hardy inequalities are based on the following rather elementary observation.

\begin{lemma}[Completeness of weakly compact sets and continuity]
 Let $q \colon H \to [0,\infty]$ be a closed quadratic form on the Hilbert space $H$ and let $C$ be a weakly compact set in $H$. 
%  Let  $\Phi \colon H \to [0,\infty]$ be a convex functional with $\Phi(0) = 0$ and let $C > 0$ such that  $B_C^\Phi = \{f \in H \mid \Phi(f) \leq C\}$ is bounded in $H$. 
 Then   $D(q) \cap C$ equipped with the pseudometric induced by the seminorm $q^{1/2}$ is complete  and for any $w \in (\ker q)^\perp$ the map
 $$D(q) \cap C \to \R, \quad f \mapsto \as{w,f}$$
 is continuous with respect to $q$.
\end{lemma}
\begin{proof}
   Let $(f_n)$ be $q$-Cauchy in $D(q) \cap C$.  By the Eberlein-Smulian theorem  $(f_n)$ has a subsequence $(f_{n_k})$ that converges weakly to some $f \in C$. The weak lower semicontinuity of $q$ yields
  $$q(f - f_n) \leq \liminf_{k \to \infty} q(f_{n_k} - f_n).$$
  This shows $f \in D(q)$ and  $f_n \to f$ with respect to $q$. 
  
 Let $w \in H$ with $w \perp \ker q$. Assume that the map $f \mapsto \as{w,f}$ is not continuous. Then there exists $\varepsilon > 0$ and $f_n,f \in D(q) \cap C$  with    $|\as{w,f - f_n}| \geq \varepsilon$ for all $n \geq 1$ and $q(f - f_n) \to 0$, as $n \to \infty$. Employing the Eberlein-Smulian theorem again, we can assume without loss of generality $f - f_n \to h$ weakly in $H$. The weak lower semicontinuity of $q$ yields
  $$q(h) \leq \liminf_{n \to \infty} q(f - f_n) = 0,$$
  so that $h \in \ker q$. But then 
  $$0 = |\as{w,h}|  = \lim_{n\to \infty}|\as{w,f_n - f}| \geq \varepsilon,$$
  a contradicition. 
\end{proof}

\begin{theorem}[Very weak Poincaré/Hardy inequality]\label{theorem:very weak}
 Let $q\colon H \to [0,\infty]$ be a closed quadratic form on the Hilbert space $H$ and let $\Phi \colon H \to [0,\infty]$ be a sublinear functional such that for every $R \geq 0$ the sublevel set $B_R^\Phi = \{f \in H \mid \Phi(f) \leq R\}$ is closed and bounded in $H$. Then for every $w \in (\ker q)^\perp$ there exists a decreasing function $\alpha = \alpha_w \colon (0,\infty) \to (0,\infty)$ such that for any $r> 0$ and all $f \in  H$ we have
 $$|\as{w,f}|^2 \leq \alpha(r) q(f) + r \Phi(f)^2. $$
\end{theorem}
\begin{proof}
   Suppose that the statement does not hold. Then there exist $r > 0$ and a sequence $(f_n)$ in $D(q)$ with
   $$|\as{w,f_n}|^2  > nq(f_n) + r \Phi(f_n)^2.$$
   In particular, $|\as{w,f_n}| > 0$. Using that $q$ is a quadratic form and that $\Phi$ is sublinear, we can assume $|\as{w,f_n}| = 1$. This implies $q(f_n) \to 0$ and $f_n \in B_{r^{-1/2}}^\Phi$. Since in Hilbert spaces closed bounded convex sets are weakly compact, the previous lemma applied to $C =B_{r^{-1/2}}^\Phi$ yields $\as{w,f_n} \to 0$, a contradiction. 
\end{proof}

\begin{remark}
 \begin{enumerate}[(a)]
  \item We call these inequalities very weak Hardy/Poincaré inequalities because in the $L^2$-case they have  $|\int_X fw d\mu|^2$ on their left  side, while for weak Hardy/Poincaré inequalities we would like to have $\int_X f^2w d\mu$ for some nonnegative $w$ on the left side. 
  \item The statement can be strengthened a bit. Since $w \in (\ker q)^\perp$, the sublinear functional $\Phi$ can be replaced by the sublinear functional 
 $$\widetilde \Phi (f)  = \inf \{\Phi(f+h) \mid h \in \ker q\}.$$
 \end{enumerate}
\end{remark}

\begin{theorem}[Weak Hardy inequality] \label{theorem:weak hardy}
 Let $q$ be a closed quadratic form on $L^2(X,\mu)$  satisfying the first Beurling-Deny criterion with $\ker q = \{0\}$. For every strictly positive $h \in L^2(X,\mu)$ and every nonnegative $w \in L^2(X,\mu) \cap L^1(X,h^2\mu)$, there exists a decreasing function $\alpha = \alpha_{w,h} \colon (0,\infty) \to (0,\infty)$ such that
 \begin{align*}
  \int_X |f|^2 w d\mu \leq \alpha(r) q(f) + r \aV{f/h}_\infty^2,\quad f \in D(q).
 \end{align*}
\end{theorem}

\begin{proof}
Consider the sublinear functional $\Phi \colon L^2(X,\mu) \to [0,\infty],\, \Phi(f) = \aV{f/h}_\infty$. It satisfies $\Phi(f) \leq C$ if and only if $|f| \leq C h$. Since $h \in L^2(X,\mu)$, this implies that the sublevel sets of $\Phi$ are bounded and closed.  Since $\ker q = \{0\}$, Theorem~\ref{theorem:very weak} applied to $w$ and $\Phi$ yields a decreasing function $\alpha \colon (0,\infty) \to (0,\infty)$ with 
\begin{align}\label{ineq:very weak}
 |\as{f,w}|^2 \leq \alpha(r) q(f) + r \aV{f/h}_\infty^2 \tag{$\heartsuit$}
\end{align}
for all $r > 0$ and $f \in L^2(X,\mu)$. Now suppose that the claimed inequality does not hold. Then there exists $r > 0$ and $(f_n)$ with $\av{f_n/h}_\infty = 1$ and 
$$\int_{X} |f_n|^2 w d\mu > n q(f_n) + r. $$
Since $|f_n| \leq h$ and $w h^2 \in L^1(X,\mu)$, this implies $q(f_n) \to 0$. Moreover, by the first Beurling-Deny criterion, we can assume without loss of generality $f_n \geq 0$ (else consider $|f_n|$).  Inequality~\eqref{ineq:very weak} then yields $f_n w \to 0$  in $L^1(X,\mu)$. We can assume without loss of generality that this convergence also holds $\mu$-a.e. (else pass to a suitable subsequence). With this at hand and $|f_n|^2 w \leq h^2w \in L^1(X,\mu)$, we deduce with Lebesgue's dominated convergence theorem
$$0 = \lim_{n \to \infty} \int_{X} |f_n|^2 w d\mu > \liminf_{n \to \infty} n q(f_n) + r > r > 0,$$
a contradiction.
% 
% 
% that $f \in L^2(X,\mu)$ with $\aV{f/h}_\infty \leq 1$. We obtain 
% %
% \begin{align*}
%  \int_X |f|^2 g dm = \int_X |f/h|^2 h^2 g dm
%  \leq \int_X |f/h| h^2 g dm
%  = |\as{|f|,hg}| \leq \alpha(r) q(|f|) + r.
% \end{align*}
% %
% With this at hand the statement follows from $q(|f|) \leq q(f)$. 
\end{proof}

\begin{remark}
\begin{enumerate}[(a)]
 \item This statement shows that for forms with the first Beurling-Deny criterion, the condition $\ker  q = \{0\}$ implies a weak Hardy inequality.  So, naturally, one may wonder whether the same is true for forms with $\ker q \neq \{0\}$. Of course, in this case, functions in $D(q)$ should be replaced by functions in  $(\ker q)^\perp$ (with respect to the $L^2(X,w\mu)$-inner product) and hence one is looking for a weak Poincaré inequality. However, examples from \cite{RW} show that weak Poincaré inequalities need not hold in general and in Section~\ref{section:weak poincare} we discuss which additional assumptions are needed. 
 
 \item Weak Poincaré inequalities were systematically introduced in \cite{RW} to characterize the rate of convergence to equilibrium for conservative Markovian semigroups without spectral gap. In our case $\ker q = \{0\}$, the associated semigroup  $(T_t)$ does not converge to an equilibrium (or projection to a ground state) but strongly to $0$, as $t \to \infty$, see e.g. \cite[Theorem~1.1]{KLWV}. Still one can ask for the corresponding rate of convergence. 
 
 If $q$ has a strictly positive excessive  function $h$, then $(T_t)$ is a contraction on $L^\infty_h(X,\mu)$. With the same arguments as in the proof of \cite[Theorem~2.1]{RW}, one can then show that the inequality
\begin{align}\label{inequality:weak poincare concrete}
 \int_X |f|^2 d\mu \leq \alpha(r) q(f) + r \aV{f/h}_\infty^2,\quad r > 0,\, f \in D(q), \tag{$\diamondsuit$}
\end{align}
with some decreasing function $\alpha \colon (0,\infty) \to (0,\infty)$ implies
$$\av{T_t f}_2^2 \leq \xi(t) (\av{f}_2^2 + \av{f/h}^2_\infty), \quad f \in L^2(X,\mu),  $$
with $\xi(t) = \inf\{r > 0 \mid -\frac{1}{2} \alpha(r) \log r \leq t\}$.  Since $\xi(t) \to 0$ as $t \to \infty$, this is a uniform rate of convergence  for the semigroup.

If $\cE$ is a Dirichlet form on $L^2(X,\mu)$ with $\mu(X) < \infty$ and $\ker \cE = \{0\}$, our theorem above applied to $w = 1$ and the excessive function $h = 1$ yields  Inequality~\eqref{inequality:weak poincare concrete} for an appropriate function $\alpha$. Note however, that our result only yields the existence of $\alpha$ but does not give an estimate for $\alpha$ (and $\xi$).
\end{enumerate}
\end{remark}

We can extend this weak Hardy inequality to all strictly positive $h$ by passing to extended spaces. In this case, the condition on the kernel is a condition on the kernel of the extended form.

\begin{corollary}\label{coro:weak hardy extended form}
 Let $q$ be a closed quadratic form on $L^2(X,\mu)$  satisfying the first Beurling-Deny criterion with $\ker q_e = \{0\}$. Let $h \in L^0(X,\mu)$ be strictly positive.  For all strictly positive $w \in L^1(X,(1+ h^2) \mu)$ 
 % 
 %$\varphi \in L^0(X,\mu)$ be strictly positive with $h \in L^2(X,\varphi \mu)$. For every  nonnegative  $g \in L^1(X,\varphi \mu) \cap  L^1(X,h^2\mu)
 there exists a decreasing function $\alpha = \alpha_{w,h} \colon (0,\infty) \to (0,\infty)$ such that for all $r > 0$
 $$\int_X |f|^2 w d\mu \leq \alpha(r) q(f) + r \aV{f/h}_\infty^2, \quad f \in D(q_e).$$
\end{corollary}
\begin{proof}
 We let $\tilde \mu = w \mu$.  The restriction of $q_e$ to $L^2(X,\tilde \mu)$ is a closed quadratic form satisfying the first Beurling-Deny criterion and it has trivial kernel.  Our assumptions on $w$  yield $1 \in L^2(X,\tilde \mu) \cap L^1(X,h^2 \tilde \mu)$. The previous theorem yields a decreasing function $\alpha$ sucht that 
 $$ \int_X |f|^2  d \tilde \mu \leq \alpha(r) q_e(f) + r \aV{f/h}_\infty^2, $$
 for all $f \in   D(q_e) \cap L^2(X,\tilde \mu)$. By our choice of $w$ the inclusion $f/h  \in L^\infty(X,\mu)$ implies $f \in L^2(X,\tilde \mu)$. Hence, the above inequality is indeed true for all $f \in D(q_e)$. Since $\tilde \mu = w \mu$, the claim follows.  
\end{proof}
% \begin{remark}
%  One important special case is when $\mu$ is finite. In this case, we can choose $h = g = \varphi = 1$ and obtain a function $\alpha \colon (0,\infty) \to (0,\infty)$ such that
%  % 
%  $$\int_X |f|^2   d\mu    \leq \alpha(r) q_e(f) + r \aV{f}_\infty^2$$
%  %
%  for all  $f \in D(q_e)$ and $r > 0$.
% \end{remark}

We finish this section with an abstract Hardy inequality that is valid for all forms satisfying the first Beurling-Deny criterion. It compares a transform of the given form with a quadratic form, which is not necessarily positive. 

\begin{proposition}[Abstract Hardy inequality]\label{prop:abstract hardy}
 Let $q$ be a closed quadratic form on $L^2(X,\mu)$ satisfying the first Beurling-Deny criterion. Let $h \in D(q_e)$ be  nonnegative and let $f \in D(q_e)$ such that $hf, hf^2 \in D(q_e)$. Then 
 $$q_e(hf) \geq q_e(hf^2,h).$$
\end{proposition}
\begin{proof}
 Without loss of generality we can assume $hf, hf^2 \in D(q)$ (else change the measure $\mu$ to $\mu'$ to make sure that $hf, hf^2 \in  D(q_e) \cap L^2(X,\mu') = D(q')$, where $q'$ is the restriction of $q_e$ to $L^2(X,\mu')$). The spectral theorem implies  
 \begin{align*}
 q(hf) - q(hf^2,h)&=  \lim_{\alpha \to 0+} \alpha \as{(I - \alpha G_\alpha) hf,hf} - \lim_{\alpha \to 0+} \alpha \as{(I - \alpha G_\alpha) hf^2,h}\\ &= \lim_{\alpha \to 0+} \alpha^2 \left( \as{G_\alpha hf^2 ,h} - \as{G_\alpha hf ,hf}\right).
 \end{align*}
 Hence, it suffices to show that $\as{ \alpha G_\alpha hf^2 ,h} - \as{ \alpha G_\alpha hf ,hf}$ is positive. Since $hf,hf^2 \in L^2(X,\mu)$ and since the resolvents are continuous, it suffices to verify positivity for simple functions $f = \sum_i \alpha_i 1_{A_i}$ with pairwise disjoint sets $A_i$ with $h1_{A_i} \in L^2(X,\mu)$. Using the symmetry of $G_\alpha$, for such a simple function we obtain 
 \begin{align*}
  \as{G_\alpha hf^2 ,h} - \as{G_\alpha hf ,hf} &= \sum_{i} \alpha_i^2 \as{G_\alpha h1_{A_i},h} - \sum_{i,j} \alpha_i \alpha_j \as{G_\alpha h1_{A_i},h1_{A_j}}\\
  &= \frac{1}{2} \sum_{i,j} \as{G_\alpha h1_{A_i},h1_{A_j}} (\alpha_i - \alpha_j)^2\\
  &\quad +  \sum_i \alpha_i^2 \as{G_\alpha h1_{A_i},h} - \sum_i \alpha_i^2 \as{G_\alpha h1_{A_i},h1_{\cup_j A_j}}.
 \end{align*}
Since $G_\alpha$ is positivity preserving, the right side of this equation is nonnegative.
\end{proof}

\begin{corollary}[Hardy inequality for perturbed forms]\label{coro:hardy inequality perturbed form}
 Let $q$ be a closed quadratic form on $L^2(X,\mu)$ satisfying the first Beurling-Deny criterion. For all strictly positive $g \in L^2(X,\mu)$ and all $f \in D(q)$  and $\alpha > 0$ we have
 $$q(f) + \alpha \av{f}_2^2 \geq \int_{X} f^2 \frac{g}{G_\alpha g} d\mu.$$
\end{corollary}
\begin{proof}
 We apply the previous proposition to the closed form $q_\alpha = q + \alpha \av{\cdot}^2$ with $D(q_\alpha) = D(q)$ and $h = G_\alpha g$. For $f \in D(q)$ we have $(f/h) h = f \in D(q)$.  Hence, if also $(f/h)^2 h \in D(q)$, the previous proposition yields 
 \begin{align*}\label{inequ:hardy}
  q_\alpha(f) = q_\alpha((f/h)h) \geq q_\alpha((f/h)^2 h,h) =  q_\alpha((f/G_\alpha g)^2 G_\alpha g,G_\alpha g) = \int_{X} f^2 \frac{g}{G_\alpha g} d\mu.
 \end{align*}

 Next we show $(f/h)^2 h \in D(q)$ for all $f \in D(q) \cap L^\infty_h(X,\mu)$.   As discussed in Lemma~\ref{lemma:existence of resolvents}, the function $h$ is $q_\alpha$ excessive and strictly positive (here we use that $G_\alpha$ is the Green operator of $q_\alpha$).  The form $q_{\alpha,h}\colon L^2(X,\mu) \to [0,\infty]$ defined by $q_{\alpha,h}(f) = q_\alpha (hf)$ is a Dirichlet form. Indeed, $h>0$ implies that $q_{\alpha,h}$ is densely defined. Moreover, $h$ is $q_\alpha$-excessive (see Remark after Lemma~\ref{lemma:characterization of excessive functions}) so that 
 $$q_{\alpha,h}(f\wedge 1) = q_{\alpha}(h(f\wedge 1)) = q_{\alpha}(hf\wedge h) \leq q_\alpha(hf) = q_{\alpha,h}(f)$$
 shows that the constant function $1$ is $q_{\alpha,h}$-excessive.  Then $D(q_{\alpha,h}) \cap L^\infty(X,\mu)$  is an algebra, see e.g. \cite[Theorem~1.4.2]{FOT}.  Since $D(q_{\alpha,h}) = \{f \in L^2(X,\mu) \mid fh \in D(q)\}$, we obtain for  $f \in D(q) \cap L^\infty_h(X,\mu)$ that $f/h \in D(q_{\alpha,h}) \cap L^\infty(X,\mu)$. The algebra property yields $(f/h)^2 \in D(q_{\alpha,h})\cap L^\infty(X,\mu)$, so that $(f/h)^2 h \in D(q)$. 
 
Now let $f \in D(q)$. What we have shown so far yields the desired inequality for the functions  $f_n = (f \wedge nh) \vee (-nh) \in D(q) \cap L^\infty_h(X,\mu).$ Using that $nh$ is strictly positive and $q_\alpha$-excessive shows
 \begin{align*}
  \int_{X} f^2 \frac{g}{G_\alpha g} d\mu &= \lim_{n \to \infty} \int_{X} f_n^2 \frac{g}{G_\alpha g} d\mu \leq  \limsup_{n \to \infty} q_\alpha(f_n) \leq q_\alpha(f). \hfill \qedhere
 \end{align*}
\end{proof}

\section{From weak Hardy inequalities to Hardy inequalities - subcriticality} \label{section:criticality}

 In this section we discuss under which conditions weak Hardy inequalities lead to Hardy inequalities, i.e., when the function $\alpha$ in the weak hardy inequality is bounded. Forms that satisfy a Hardy inequality are called subcritical in the literature. Hence, the content of this section is devoted to characterizing subcriticality for quadratic forms satisfying the first Beurling-Deny criterion. Our arguments will rely  on our weak and abstract Hardy inequalities and their  corollaries.

\begin{definition}[Subcriticality]
A quadratic form $q$ on $L^2(X,\mu)$ is called {\em subcritical} if there exists a strictly positive $w \in L^0(X,\mu)$ such that the following {\em Hardy inequality} holds
$$\int_X |f|^2 w d\mu \leq q(f), \quad f \in D(q).$$
In this case, $w$ is called {\em Hardy weight for }$q$. 
\end{definition}

\begin{remark}
 The definition of $q^+$ and $q_e$ and Fatou's lemma yield that Hardy inequalities as in the last definition extend to $D(q^+)$ and $D(q_e)$. In particular, for subcritical $q$ we have $D(q^+) = D(q_e) \cap L^0_+(X,\mu)$. 
\end{remark}

The aim of this section is to prove the following theorem. 

\begin{theorem}[Characterization subcriticality]\label{theorem:subcriticality}
Let $q$ be a closed quadratic form on $L^2(X,\mu)$ satisfying the first Beurling-Deny criterion. The following assertions are equivalent. 
\begin{enumerate}[(i)]
 \item $q$ is subcritical. %
 %  \item There exists $g> 0$ such that $D(q_e) \subseteq L^1(X,g\mu)$.
 \item There exists a strictly positive $g\in L^0(X,\mu)$ such that for all $f \in D(q)$
 $$\int_X |f|g d\mu \leq q(f)^{1/2}.$$
 \item $\ker q^+ = \{0\}$. 
 \item $\ker q_e = \{0\}$ and there exists a strictly positive $q$-excessive function $h \in D(q_e)$.
 \item For one  strictly positive $f \in L^0(X,\mu)$ the limit $Gf = \lim_{\alpha \to 0+} G_\alpha f$ exists in $L^0(X,\mu)$.
  \item For all strictly positive $q$-excessive functions $h \in L^0(X,\mu)$ and all  $f \in L_+^1(X,h \mu)$ the limit $Gf = \lim_{\alpha \to 0+} G_\alpha f$ exists in $L^0(X,\mu)$. 
 \item $\ker q_e = \{0\}$  $q \neq 0$ and the map 
 $$D(q_e) \to D(q_e),\quad f \mapsto |f|$$
 is continuous with respect to $q_e$. 
 \item The embedding $(D(q_e),q_e) \to L^0(X,\mu),\, f \mapsto f$ is continuous. 
 \item $(D(q_e),q_e)$ is a Hilbert space. 
\end{enumerate} 
\end{theorem}

\begin{proof} {\em Step 1:} We first discuss the equivalence of (iii), (iv), (vii), (viii) and (ix). 

 (viii) $\Leftrightarrow$ (ix): This is the content of Theorem~\ref{theorem:2 imply the third}.

 (viii) \& (ix) $\Rightarrow$ (iv):  Let $g\in D(q_e)$ with $g> 0$ be given.  By (viii) the set $\{f \in D(q_e) \mid f \geq g\}$ is a closed nonempty convex set in the Hilbert space  $(D(q_e),q_e)$. Hence, by the approximation theorem in Hilbert spaces, there exists $h \in D(q_e)$ with
 $$q_e(h) = \inf \{q_e(f) \mid f \geq g\}.$$
 For nonnegative $\varphi \in D(q_e)$ and $\varepsilon > 0$ we have $h + \varepsilon \varphi \geq g$ such that 
 $$q_e(h) \leq q_e(h + \varepsilon \varphi) = q_e(h) + 2\varepsilon q_e(h,\varphi) + \varepsilon^2 q_e(\varphi).$$
 Letting $\varepsilon \to 0+$ yields $q_e(h,\varphi) \geq 0$, so that $h$ is $q_e$-excessive by Lemma~\ref{lemma:characterization of excessive functions} (choose a measure $\mu' = \varphi \mu$ such that $h \in L^2(X,\mu')$ and the result follows from Lemma~\ref{lemma:characterization of excessive functions} applied to the restriction of $q_e$ to $L^2(X,\mu')$).
 
 Since $L^0(X,\mu)$ is Hausdorff, the continuity of the embedding $D(q_e) \to L^0(X,\mu)$ implies that $(D(q_e),q_e)$ is Hausdorff as well. This means $\ker q_e = \{0\}$.

% 
% We first show part of (iv), namely $\ker q_e  = \{0\}$. Suppose $\ker q_e \neq \{0\}$ so that there exists a nonnegative $0 \neq h \in \ker q_e$. For an $f \in D(q_e)\setminus \ker q_e$ with $0 \leq f \leq h$ (such an $f$ always exists since $q_e \neq 0$) we obtain $f - nh \to f$ with respect to $q_e$ but  
%  %
%  $$q_e(|f - nh| - |f|) = q_e(nh - 2f) = q_e(2f).$$
%  %
%  This implies that $\lvert\cdot\rvert$ is not continuous,   a contradiction to our assumption.  
  (vii) $\Rightarrow$  (viii): Let  $h \in D(q_e)$ be strictly positive.  Since $\ker q_e =\{0\}$, Corollary~\ref{coro:weak hardy extended form} yields a function $\alpha \colon (0,\infty) \to (0,\infty)$ and $w > 0$ such that 
 $$\int_X (|f|\wedge h)^2w  d\mu \leq \alpha(r)q_e(|f| \wedge h) + r  $$
 for all $r> 0$ and $f \in D(q_e)$. Let  $(f_n)$ in $D(q_e)$ with $q_e(f_n) \to 0$ be given. The continuity of $\lvert\cdot\rvert$ with respect to $q_e$ implies that the map  $D(q_e) \to D(q_e), \,f \mapsto  |f|\wedge h$ is also continuous with respect to $q_e$. We obtain $q_e(|f_n| \wedge h) \to 0$ and our weak Hardy inequality implies
 $$\int_X (|f_n|\wedge h)^2 w  d\mu \to 0.$$
This yields $f_n \to 0$ in $L^0(X,\mu)$ and we arrive at (viii).  
 
(viii) $\Rightarrow$ (vii): Assertion (viii) implies that the form topology is induced by the norm $q_e^{1/2}$. With this at hand, the assertion follows from Lemma~\ref{lemma:ancona}.
 
(iv) $\Rightarrow$ (iii): Let $f \in \ker q^+$ and let $f_n \in D(q_e)$ with $f_n \to f$ and $q_e(f_n) \to 0$.  Let $h \in D(q_e)$ be excessive and strictly positive.  Using $f_n \wedge h \to f \wedge h$ in $L^0(X,\mu)$ and the lower semicontinuity of $q_e$  we obtain 
$$q_e(f \wedge h) \leq \liminf_{n\to \infty}  q_e (f_n \wedge h) \leq \liminf_{n\to \infty}  q_e (f_n) = 0.$$
This implies $f \wedge h \in \ker q_e$, so that $f \wedge h = 0$ by our assumption. The strict positivity of $h$ yields $f = 0$.

(iii) $\Rightarrow$ (viii): Let $(f_n)$ be in $D(q_e)$ with $q_e(f_n) \to 0$. Since $(f_n)$ is an arbitrary sequence with this property, it suffices to show that $(f_n)$ has a subsequence converging to $0$  $\mu$-a.e. Using the first Beurling-Deny criterion, we choose a subsequence $(f_{n_k})$ such that 
$$\sum_{k = 1}^\infty q^+(|f_{n_k}|) = \sum_{k = 1}^\infty q_e(|f_{n_k}|) < \infty.$$
With this and $\ker q^+ = \{0\}$ at hand, $|f_{n_k}| \to 0$ $\mu$-a.e. follows from Lemma~\ref{lemma:continuity q^+}. 

% 
% 
% 
% We prove that 
% %
% $$g = \limsup_{k \to \infty} |f_{n_k}| = \lim_{k \to \infty}  \sup_{l \geq k} |f_{n_l}| = \lim_{k \to \infty} \lim_{N \to \infty}  \sup_{ N\geq l \geq k} |f_{n_l}| \in L^+(X,\mu) $$
% %
% vanishes by showing $g \in \ker q^+ = \{0\}$. This then implies $f_{n_k} \to 0$ in $L^0(X,\mu)$.
% 
% Since $q^+$ is lower semicontinuous on $L^+(m)$ and since  the limits in the definition of $g$ hold in $L^+(m)$, we obtain
% %
% \begin{align*}
%  q^+(g) &\leq \liminf_{k \to \infty} \liminf_{N \to \infty} q^+\left(\sup_{ N\geq l \geq k} |f_{n_l}|\right) \\
%  &= \liminf_{k \to \infty} \liminf_{N \to \infty} q_e\left(\sup_{ N\geq l \geq k} |f_{n_l}|\right)\\
%  &\leq  \liminf_{k \to \infty} \liminf_{N \to \infty} \sum_{l = k}^N q_e (|f_{n_l}|)\\
%  &\leq  \liminf_{k \to \infty} \sum_{l = k}^N q_e(f_{n_l}) = 0.
% \end{align*}
% %
% For the second inequality we used  subadditivity of $q_e$ with respect to taking the maximum of functions. 

{\em Step 2:} Now that we established the equivalence of  (iii), (iv), (vii), (viii) and (ix), we show that all of these are equivalent to (ii), (v) and (vi).

(vi) $\Rightarrow$ (v): This is trivial. 

(v) $\Rightarrow$ (ii):   Without loss of generality we can assume that $f$ is strictly positive and in $L^1(X,\mu) \cap L^2(X,\mu)$. Consider the function $F = f/(Gf \vee 1)$. Then $F \leq f$ and $F \leq f/Gf$, and so the monotonicity of $G$ yields $\int_X F GF d\mu \leq \av{f}_1.$ This implies 
$$q_\alpha(G_\alpha F) = \as{F,G_\alpha F} \leq \av{f}_1. $$
For $g \in D(q)$ we obtain %
$$\as{|g|,F} = q_\alpha(|g|,G_\alpha F) \leq q_\alpha(|g|)^{1/2}q_\alpha(G_\alpha F)^{1/2} \leq \av{f}_1^{1/2}  q_\alpha(g)^{1/2}. $$
Letting $\alpha \to 0+$ yields (ii). 

(ii) $\Rightarrow$ (iv): If (ii) holds we  have $\ker q_e = \{0\}$ (since the inequality extends to the extended space, see the remark before this theorem). The existence of a strictly positive excessive function und (ii) is the content of Lemma~\ref{lemma:existence of resolvents}~(b).

(iv) $\Rightarrow$ (vi): Let $h\in L^0(X,\mu)$ be striclty positive and $q$-excessive and let $f \in L^1_+(X,h\mu)$. Without loss of generality we can assume $h \in D(q_e)$ (else consider $h \wedge h'$ for a strictly positive excessive function $h' \in D(q_e)$). Lemma~\ref{lemma:extension of resolvent to extended space} and (viii) imply $\alpha G_\alpha h \to 0$ in $L^0(X,\mu)$ as $\alpha \to 0+$. The resolvent identity and $\alpha G_\alpha h \leq h$, see Lemma~\ref{lemma:characterization of excessive functions},  show  that this convergence is monotone and that 
$$G_\beta (h - \alpha G_\alpha h) = G_\alpha h - \beta G_\alpha G_\beta h \leq G_\alpha h \leq \alpha^{-1}h.$$

Let $A_\alpha = \{\alpha G_\alpha h \leq h/2\}$ and note that $A_\alpha \nearrow X$ as $\alpha \to 0+$ by our previous considerations. Hence, it suffices to show that $Gf$ is a.s. finite on $A_\alpha$.  Using symmetry of the extended resolvents and the inequality above, for  $\beta > 0$ we estimate
\begin{align*}
 \int_{A_\alpha} G_\beta f h d\mu  \leq 2 \int_{X} f G_\beta  (h - \alpha G_\alpha h) d\mu = \frac{1}{\alpha}  \int_{X} fh d\mu.
\end{align*}
Hence, we obtain  $1_{A_{\alpha}} Gf \in L^1(X,h\mu)$ and arrive at (ii).

{\em Step 3:} Assertions (ii) to (ix) are equivalent to (i):

 (i) $\Rightarrow$ (ii): Without loss of generality we can assume that the Hardy weight $w$ is in $L^1(X,\mu)$. For $f \in D(q)$ we obtain
% %
$$\int_X |f| w d\mu \leq \left(\int_X |f|^2 w d\mu \right)^{1/2} \left(\int_X   w d\mu \right)^{1/2} \leq C q(f)^{1/2}. $$
This yields the claim. 

(v) $\Rightarrow$ (i): Without loss of generality we can assume $Gg \in L^0(X,\mu)$ for some strictly positive $g \in L^2(X,\mu)$. Then Corollary~\ref{coro:hardy inequality perturbed form} yields the claim for the Hardy weight $w = g/Gg$ after letting $\alpha \to 0+$.  
\end{proof}

\begin{remark}[Hardy weights]
 The proof of the theorem shows that for all strictly positive $g \in L^0(X,\mu)$ with $Gg \in L^0(X,\mu)$ the function $w = g/Gg$ is  a Hardy weight for $q$. A criterion for the existence of $Gg$ is given by assertion (iv).
\end{remark}

\begin{remark}[State of the art]
 \begin{enumerate}[(a)] 
\item Let $\cE$ be a Dirichlet form. In this case, the constant function $1$ is $\cE$-excessive and assertion (iii) reduces to $\ker \cE_e = \{0\}$.   This is one characterization of transience of the Dirichlet form $\cE$ and the equivalence of transience to (ii), (iii)  (iv), (v), (vi) and (ix) is well-known, see the discussion in \cite[Section~1.6 and Notes]{FOT}.  The equivalence of transience to subcriticality with Hardy weight $w = g/Gg$ is more or less contained in \cite{Fit} for regular Dirichlet forms. The additional regularity assumption in \cite{Fit} allows that the measure $w\mu$ on one side of the Hardy inequality can even be replaced by a smooth measure. The equivalence of transience to  (viii) is based on Theorem~\ref{theorem:2 imply the third} and taken from \cite{Schmi3}.

\item For second-order linear elliptic operators the relation of Hardy inequalities and the behavior of the resolvent at the infimum of the spectrum (which in our case can always be taken to equal $0$), i.e., the equivalence of (i) and (v), (vi) is well known, see e.g.  \cite{Mur,Pin,Pins} and references therein.

The connection of subcriticality to transience of  $h$-transformed  Schrödinger type forms (Dirichlet form + form induced by a potential) is  studied in the recent \cite{TU}, with previous results in \cite{Tak1,Tak2}. For Schrödinger type forms on discrete spaces corresponding results were obtained in \cite{KPP}.

\item The implication  (viii) \& (ix) $\Rightarrow$ (iv) uses a standard argument showing that certain minimizers of the 'energy' $q$ are 'superharmonic' functions. The idea to use the function $f/(Gf \vee 1)$ in the proof of implication (v) $\Rightarrow$ (ii) is taken from the proof of \cite[Theorem~1.5.1]{FOT}.

\item The equivalence of subcriticality to (iii) and (vii) seems to be a new observation. 
\end{enumerate}
\end{remark}

With this at hand we  obtain that there are three types of irreducible forms satisfying the first Beurling-Deny criterion.

\begin{corollary}\label{corollary:trichotomy} 
 Let $q$ be an irreducible quadratic form on $L^2(X,\mu)$ satisfying the first Beurling-Deny criterion. Then precisely one the following assertions holds.
 \begin{enumerate}[(i)]
  \item $q$ is subcritical. 
  \item $\ker q_e = \R h$ for some strictly positive $h \in L^0(X,\mu)$.
  \item $\ker q^+ = \{0,\infty\}$.
 \end{enumerate}
Moreover, (iii) holds if and only if $q$ does not possess a nontrivial nonnegative excessive function. 
\end{corollary}
\begin{proof}
 Suppose (i) does not hold. According to our theorem $\ker q^+ \neq \{0\}$. Now there are two cases:
 
 Case~1: $\ker q^+ \cap L^0(X,\mu) \neq \{0\}$. Since $D(q_e) \cap L^0_+(X,\mu) = D(q^+) \cap L^0_+(X,\mu)$, this implies $\ker q_e \neq \emptyset$. Now (ii) follows from Corollary~\ref{coro:perron frobenius}. Since $\ker q_e \cap L^0_+(X,\mu) \subseteq \ker q^+$, this implies that (iii) does not hold.
 
 Case~2: $\ker q^+ \cap L^0(X,\mu) = \{0\}$. Let $0 \neq h \in \ker q^+$. Since $h \not\in L^0(X,\mu)$, the set $A = \{h = \infty\}$ has positive measure. We show that it is $q_e$-invariant as this implies (iii) by the irreducibilty of $q$. 
 
According to Lemma~\ref{lemma:kernel q^+}, we have  $ 1_A \cdot \infty \in \ker q^+$. For nonnegative $f \in D(q_e)$ we obtain 
$$q_{e}(1_{A} f ) = q_e (f \wedge (1_A \cdot \infty)) = q^+(f \wedge (1_A \cdot \infty)) \leq q^+(f) + q^+( 1_A \cdot \infty) = q_e(f). $$
Since $q_e$ satisfies the first Beurling-Deny criterion, this implies $f 1_A\in D(q_e)$ for all $f \in D(q_e)$ and 
\begin{align*}
 q_{e}(1_Af ) &= q_e(1_{A} f_+ ) +  q_e(1_{A}f_- ) -  2 q_e(1_A f_+ ,1_A f_-)\\
 &\leq q_e(f_+) + q_e(f_-) - 2 q_e(1_A f_+, 1_A f_-).
\end{align*}
It remains to prove $q_e(1_A f_+ ,1_A f_- ) \geq q_e(f_+,f_-)$ to establish the invariance of $A$. Nonnegative $h,h' \in D(q_e)$ with $h \wedge h' = 0$ satisfy $q_e(h,h') \leq 0$ (this is a consequence of Lemma~\ref{lemma:form domain lattice}). Hence,  
\begin{align*}
 q_e(f_+,f_-) &= q_e(1_{A} f_+,1_{A}f_-) + q_e(1_{A} f_+,1_{X \setminus A} f_-) + q_e(1_{X \setminus A} f_+, 1_A f_-) + q_e(1_{A} f_+, 1_{X \setminus A} f_-)\\
 & \leq  q_e(1_{A} f_+,  1_A f_-). 
\end{align*}

It remains to prove the 'Moreover' statement. If (iii) does not hold, either (i) or (ii) are satisfied. If (i) holds, then $q$ has an excessive function by the previous theorem and if (ii) holds, $h$ is an excessive function.

Now suppose $q$ has a nontrivial  excessive function $h \in L^0_+(X,\mu)$. If $\infty$ were  in $\ker q^+$, we would obtain 
$$q^+(h) = q^+(\infty \wedge h) \leq q^+(\infty) = 0.$$
This implies $\ker q_+ \neq \{0,\infty\}$, so that (iii) does not hold.
\end{proof}

\begin{definition}[Criticality]
 An irreducible form satisfying (ii) in the previous corollary is called {\em critical} and the strictly positive function $h$ is an {\em Agmon ground state} of $q$.
\end{definition}

\begin{remark}  For irreducible Dirichlet forms the dichotomy between subcriticality (transience) and criticality (recurrence) is well-known, see \cite[Section~1.6]{FOT}. It is also known for classical Schrödinger operators \cite{PT}, certain generalized Schrödinger forms \cite{Tak1} and discrete Schrödinger operators \cite{KPP}. 
 
 For general irreducible forms however, it may happen happen that they are neither critical nor subcritical. According to the corollary, this is precisely the case if they do not possess  excessive functions. 
 
 In concrete models the existence of an excessive function is often known. Indeed, we do not have a counterexample of a closed form satisfiying the first Beurling-Deny criterion without excessive function. Excessive functions usually correspond to superharmonic functions with respect to an associated 'weakly defined operator'.   Existence results for such functions are sometimes referred to as Allegretto-Piepenbrink type theorems, see e.g.  \cite{Sim84,LSV,HK} and references therein. Below in Appendix~\ref{section:appendix} we prove existence of excessive functions for irreducible forms for which the semigroup admits a heat kernel, an assumption that is satisfied for the models considered in the mentioned \cite{PT,Tak1,KPP}.  Our existence result relies on a weak Harnack principle, which is shown to hold for kernel operators with strictly positive kernel.
\end{remark}

% \begin{lemma}
%  Let $q$ be a closed quadratic form on $L^2(X,\mu)$ satisfiying the first Beurling-Deny criterion with $\ker q_e = \{0\}$ and suppose that $h \in  D(q_e)$ is a nonnegative excessive function.  Then $\alpha G_\alpha h \searrow  0$, as $\alpha \to 0+$.
% \end{lemma}
% 
% % 
% % 
% % 
% \begin{theorem}
%   Let $q$ be a closed quadratic form on $L^2(X,\mu)$  satisfying the first Beurling-Deny criterion and let $\Phi \colon L^2(X,\mu) \to [0,\infty]$ be a sublinear functional such that for every $R \geq 0$ the sublevel set $B_R^\Phi = \{f \in H \mid \Phi(f) \leq R\}$ is closed and bounded in $L^2(X,\mu)$. If $\ker q = \{0\}$, then for every $0 \leq g \in L^1(m) \cap L^2(X,\mu)$ there exists an increasing function $\alpha = \alpha_g \colon (0,\infty) \to (0,\infty)$ such that for any $r> 0$ and all $f \in  L^2(X,\mu)$ we have
%  %
%  $$\int_X |f|^2 g dm \leq \alpha(r) q(f) + r \Phi(f)^2. $$
%  %
% \end{theorem}

\section{Weak Poincaré inequalities and completeness of extended form domains} \label{section:weak poincare}
In this section we discuss when weak Poincaré inequalities hold. For conservative Dirichlet forms on $L^2(X,\mu)$ with finite $\mu$, they have been introduced and extensively studied in \cite{RW}, which also contains an abundance of examples and further references. Therefore, here we restrict ourselves to two additional abstract criteria for the validity of such an inequality.

For simplicity we assume irreducibility of the form even though this is not necessary. Moreover, we write $f \perp_w h$ to state that the functions $f$ and $h$ are orthogonal in the Hilbert space $L^2(X,w\mu)$.

\begin{theorem}\label{theorem:completeness extended space critical case}
 Let $q$ be an irreducible closed quadratic form on $L^2(X,m)$ satisfying the first Beurling-Deny criterion. Assume further that $q$ is critical such that  $\ker q_e = \R h$ for some strictly positive $h \in L^0(X,\mu)$. The following assertions are equivalent. 
 \begin{enumerate}[(i)]
  \item $(D(q_e)/\R h, q_e)$ is a Hilbert space. 
  \item The map 
  $$(D(q_e)/\R h, q_e) \to L^0(X,\mu)/\R h, \quad f + \R h \mapsto f  +  \R h $$
  is continuous. Here, $L^0(X,\mu)/\R h$ is equipped with the quotient topology.
  \item For one/all strictly positive $w \in L^1(X,h^2\mu)$ there exists a decreasing  $\alpha \colon (0,\infty) \to (0,\infty)$ such that for all $f \in L^2(X,w\mu)$ with $f \perp_w h$ and all $r > 0$ we have
  $$\int_X f^2 w d\mu \leq \alpha(r) q_e(f) +r \av{f/h}_\infty^2. $$
  %
%   \item    For one/all strictly positive $g \in L^1(X,h^2\mu)$ there exists an increasing  $\alpha \colon (0,\infty) \to (0,\infty)$ such that for all $f \in L^2(X,g\mu)$ and all $r > 0$ we have
%   %
%   $$\int_X (f-Pf)^2 g d\mu \leq \alpha(r) q_e(f) +r \delta(f/h)^2.$$
%   %
%   Here, $Pf = \av{g}_1^{-1} \int f h g d\mu$ denotes the orthogonal projection to $\R 1$ in $L^2(X,g\mu)$ and $\delta(f) = \esssup f - \essinf f$. 
 \end{enumerate}
\end{theorem}
\begin{proof}
 The equivalence of (i) and (ii) is the content of Theorem~\ref{theorem:2 imply the third}.
 
 (iii) $\Rightarrow$ (ii): It suffices to show that if $(f_n)$ is a sequence in $D(q_e)$ with $q_e(f_n) \to 0$, then there exist $C_n \in \R$ such that $f_n - C_n h \to 0$ in $L^0(X,\mu)$. Let $w \in L^1(X,h^2\mu)$ be  strictly positive  such that the weak Poincaré inequality holds and consider
 $$T \colon L^0(X,\mu) \to L^2(X,w\mu), \quad Tf = (f \wedge h) \vee(-h).$$ 
 For each $n \in \N$ there exists $C_n \in \R$ such that $T(f_n-C_n h) \perp_w h$. Since also $|Tf| = |f| \wedge h$, the weak Poincaré inequality applied to $T(f_n-C_n h)$  and $h \in \ker q_e$ imply 
 $$\int_X (|f_n - C_n h| \wedge h)^2 w d\mu \leq \alpha(r) q_e(T(f_n-C_n h)) + r \leq \alpha(r) q_e(f_n) + r, \quad r >0.$$
 Since $h,w$ are strictly positive, we infer $f_n - C_n h \to 0$ in $L^0(X,\mu)$ from $q_e(f_n) \to 0$.
 
 (ii) $\Rightarrow$ (iii): Suppose that the weak Poincaré inequality does not hold for a given strictly positive $w \in L^1(X,h^2\mu)$. Then there exist $r > 0$ and a sequence $(f_n)$ in $D(q_e) \cap L^\infty_h(X,\mu)$ with $f_n \perp_w h$ such that $\int f_n^2 w d\mu = 1$ and 
 $$1 > n q_e(f_n) + r^2 \av{f_n/h}_\infty^2.$$
 This implies $|f_n| \leq h/r$ and $q_e(f_n) \to 0$. Assertion (ii) yields $C_n \in \R$ such that $f_n  - C_n h \to 0$ in $L^0(X,\mu)$. Without loss of generality we assume the convergence holds $\mu$-a.e. (else pass to a suitable subsequence). It suffices to show  $C_n\to 0$ and hence $f_n \to 0$ $\mu$-a.e. Indeed, using $|f_n| \leq h/r$ and $h \in L^2(X,w\mu)$,  Lebesgue's dominated convergence theorem and $f_n \to 0$ $\mu$-a.e. imply
 $$0 = \lim_{n \to \infty} \int_X f_n^2 w d\mu = 1,$$
 a contradiction.
 
 From $h > 0$ and 
 $$|C_n h| \leq |C_n h - f_n| + |f_n| \leq |C_n h - f_n| + h/r \to h/r, \quad n \to \infty,$$
 it follows that $C_n$ is bounded. Hence,  $f_n - C_n h$ is bounded by some constant times $h$.  Using this and $f_n \perp_w h$, we conclude with the help of  Lebesgue's dominated convergence theorem
 $$-C_n \int_X h^2 w d\mu = \int_X (f_n - C_n h) h w d\mu \to 0, \quad n \to \infty.$$
 This shows the required $C_n \to 0$. 
\end{proof}

\begin{remark}
\begin{enumerate}[(a)]
 \item Instead of working with the norm $\av{\cdot/h}_\infty$ on the right side of the inequality in (iii), we could have also used $\delta_h(f) = \esssup f/h - \essinf f/h$. Note that $\delta_h$ is the quotient norm on $L^\infty_h(X,\mu)/\R h$. 
 \item It is  known for recurrent (critical) irreducible Dirichlet forms $\cE$ that Poincaré type inequalites yield that $(D(\cE_e)/\R 1,\cE_e)$ is a Hilbert space, see e.g. the discussion in \cite[Section~4.8]{FOT}, which is based on \cite{Osh} and treats Harris recurrent Dirichlet forms.  The observation that the converse can be characterized by a weak Poincaré inequality seems to be new.  
 
  Recall Theorem~\ref{theorem:subcriticality}, which states that if $(D(q_e),q_e)$ is a Hilbert space (and in particular $\ker q_e = \{0\}$), not only weak Hardy inequalites but Hardy inequalites hold with respect to certain Hardy weights. It would be interesting to know whether Poincaré inequalites with respect to certain weights hold under the condition that $(D(q_e)/\ker q_e,q_e)$ is a Hilbert space or what else has to be assumed. As mentioned above, for Dirichlet forms Harris recurrence is sufficient for a Poincaré inequality to hold. 
  
  In some sense  a Poincaré inequality for $q$ can be interpreted as subcriticality of the form $q$ considered on the quotient space $L^2(X,\mu)/\R h \simeq (\R h)^\perp$ (or $q_e$ on the quotient $L^0(X,\mu)/\R h$). These quotients do not carry a good order structure. Hence, the methods used in the proof of Theorem~\ref{theorem:subcriticality}, which heavily rely on the order structure of the function spaces, are not available.
%   
%   
%   The proof of the Hardy inequality with weight $w = g/Gg$ heavily relies on the order structure of the function spaces, which is not available anymore after taking the quotient mod $\R h$.
\end{enumerate}
\end{remark}

For an irreducible conservative Dirichlet form $\cE$ and finite  $\mu$, in  \cite{RW} it is noted that the validity of a weak Poincaré inequality (with $h = w = 1$) is equivalent to Kusuoka-Aida's weak spectral gap property discussed in \cite{Aid,Kus}. The {\em weak spectral gap property} is said to hold if   sequences $(f_n)$ in $D(\cE)$ with $\av{f_n}_2 \leq 1$, $f_n \perp 1$ and $\cE(f_n) \to 0$ converge to $0$ in measure with respect to $\mu$. In this sense, our main observation is that the weak spectral gap property is the same as the continuity of the embedding in assertion (ii). As shown in \cite[Lemma~2.6]{Aid}, the weak spectral gap property holds if the semigroup is a semigroup of kernel operators. Hence, we obtain the following corollary.

\begin{corollary}
Let $\cE$ be an irreducible conservative Dirichlet form on $L^2(X,\mu)$ with finite  $\mu$. If the associated semigroup admits an integral kernel, then $(D(\cE_e)/\R 1,\cE_e)$ is a Hilbert space. 
\end{corollary}

  \cite[Theorem~7.1]{RW} gives a sharp criterion for the weak Poincaré inequality for a conservative irreducible Dirichlet form on a configuration space over a non-compact manifold $M$ (with a finite measure $\mu$ on the configuration space, $h = w = 1$ and with respect to the quotient norm $\delta(f) = \esssup f - \essinf f$ on the right side of the weak Poincaré inequality). In particular, it shows that if 
  $$\lambda(r) = \inf \{\av{\nabla f}_{\vec{L}^2(M)} \mid f \in C_c^\infty(M) \text{ with }  \av{f}_{L^2(M)} = 1 \text{ and } \av{f}_\infty^2 \leq r \} = 0 $$
 for some $r > 0$, then the considered Dirichlet form on the configuration space over $M$ does not satisfy a weak Poincaré inequality. But for $M = \R$ it is readily verified that $\lambda(r) = 0$ for any $r > 0$ (for $n \in \N$ consider smoothed  versions of $f_n \colon \R \to \R,\, f_n = -|x|/n^2 + 1/n$). Hence, on $M = \R$ the Dirichlet forms on configuration space over $\R$ described in \cite[Section~7]{RW} do not satisfy weak Poincaré inequalities and according to  our theorem their extended Dirichlet space is not complete. The existence of such examples seems to be a new observation and therefore we state it as a corollary to our theorem. 
 
 \begin{corollary}\label{coro:non complete}
  There exists an irreducible Dirichlet form $\cE$ with $\ker \cE_e = \R 1$, such that $(D(\cE_e)/\R 1,\Ee)$ is not a Hilbert space.
 \end{corollary}

%  
%  
% 
% an example of a Dirichlet form on Loop/configuration space without weak Poincaré inequality (for the function $h =1$ and $g =1$), see Example~....  Hence, according to our theorem, the extended Dirichlet space of this form modulo constant functions is not a Hilbert space.  To the best of our knowledge it is a new observation that there are recurrent irreducible Dirichlet forms for which $(D(\cE_e)/\R 1,\cE_e)$ is not complete. This stands in contrast to the transiet case $\ker \cE_e = \{0\}$, where $(D(\cE_e),\cE_e)$ is always complete. 

\appendix

\section{Existence of excessive functions and a question of Schep}\label{section:appendix}
In this appendix we show that an irreducible form with the first Beurling-Deny criterion has a strictly positive excessive function if its semigroup is a semigroup of kernel operators. This result is based on a more general observation for positivity preserving kernel operators on $L^p$-spaces, which answers a question of Schep for this particular class of operators.

% 
% \begin{lemma}\label{lemma:lower sc}
%  Let $T \colon L^p(X,\mu) \to L^p(Y,\nu)$ be a positive operator and let $(f_n)$ be a sequence in $L^p_+(X,\mu)$. Then 
%  %
%  $$T(\liminf_{n \to \infty} f_n) \leq \liminf_{n \to \infty} T f_n.$$
%  %
% \end{lemma}
% \begin{proof}
%  This follows directly from the definitions. 
% \end{proof}
%

In this section $(Y,\nu)$ denotes another $\sigma$-finite measure space. We assume $1 < p < \infty$ and consider positivity preserving operators $T \colon L^p(X,\mu) \to L^p(Y,\nu)$, which are automatically continuous, see e.g. \cite[Proposition~1.3.5]{MN}. Their adjoint is denoted by $T^* \colon L^q(Y,\nu) \to L^q(X,\mu)$ with $1/q + 1/p  = 1$. Note that if $f \in L_+^p(X,\mu)$, then $f^{p-1} \in L^q(X,\mu)$. In particular, if $T$ is positivity preserving, $T^*(Tf)^{p-1}$ is well-defined for $f \in L^p_+(X,\mu)$ and belongs to $L^q(X,\mu)$. 

We consider the quantity
$$\lambda(T) = \inf \{\lambda \geq 0 \mid \text{ ex.  strictly positive } f\in L_+^p(X,\mu) \text{ with }T^* (T f)^{p-1} \leq \lambda f^{p-1}\}. $$
It turns out that $\lambda(T) = \av{T}^p$, see \cite[Theorem~4 and Theorem~8]{Schep}, but we shall not use this fact. 

By definition, for $\lambda  > \lambda(T)$ there exists a strictly positive $f \in L^p_+(X,\mu)$ with $T^* (T f)^{p-1} \leq \lambda f^{p-1}$. Schep asked in \cite[Section~8]{Schep}  what happens at  $\lambda = \lambda(T)$. In this case, one cannot expect to find a corresponding $f \in L_+^p(X,\mu)$, but can still hope for a strictly positive $f \in L^0_+(X,\mu)$. The following theorem shows that this is indeed true if $T$ is a kernel operator with strictly positive integral kernel.   

\begin{theorem}\label{theorem:main}
 Let $1 < p < \infty$ and let $T \colon L^p(X,\mu) \to L^p(Y,\nu)$ be a kernel operator with strictly positive integral kernel. Then there exists a strictly positive  $h \in L^0(X,\mu)$ such that
$$T^* (T h)^{p-1} \leq \lambda(T) h^{p-1}.$$
\end{theorem}

This theorem and our existence result for excessive functions are a consequence of the following lemma, whose proof we give at the end of this section.  

\begin{lemma} \label{lemma:main excessive}
  Let $1 < p < \infty$ and let $T \colon L^p(X,\mu) \to L^p(Y,\nu)$ be a kernel operator with strictly positive integral kernel. Let $\lambda \geq 0$ and let $f_n \in L^p(X,\mu)$ be strictly positive with $T^* (T f_n)^{p-1} \leq \lambda  f_n^{p-1}$. Then there exist $C_n > 0$ such that $h = \liminf_{n \to \infty} C_n f_n$ satisfies $0 < h < \infty$  $\mu$-a.e. and 
  $$T^* (T h)^{p-1} \leq \lambda  h^{p-1}.$$
  \end{lemma}

\begin{proof}[Proof of Theorem~\ref{theorem:main}]
 Choose a sequence $\lambda_n \searrow \lambda(T)$ and let $f_n \in L^p(X,\mu)$ be strictly positive with  $T^* (T f_n)^{p-1} \leq \lambda_n f_n^{p-1}$. Choose $C_n$ according to Lemma~\ref{lemma:main excessive} and set $h = \liminf_{n \to \infty} C_n f_n$. Then $h$ is striclty positive and in $L^0(X,\mu)$. Fatou's lemma for positivity preserving operators yields
 \begin{align*}
  T^* (T h)^{p-1} &\leq \liminf_{n \to \infty}  T^* (T (C_nf_n))^{p-1} \leq \liminf_{n \to \infty} \lambda_n (C_nf_n)^{p-1} = \lambda(T)h^{p-1}. \hfill \qedhere
 \end{align*}
\end{proof}

\begin{theorem}\label{theorem:existence of excessive functions}
 Let $q$ be an irreducible quadratic form on $L^2(X,\mu)$ satisfying the first Beurling-Deny criterion. Assume that the associated semigroup is a semigroup of kernel operators. Then there exists a strictly positive $q$-excessive function.    
\end{theorem}
\begin{proof}
 Irreducibility yields that the integral kernel of the associated semigroup $(T_t)$ is strictly positive. We choose a strictly positive function $g > 0$. Then $f_\alpha = G_\alpha g$ is strictly positive and $T_t f_\alpha \leq e^{t\alpha} f_\alpha$ (use that $G_\alpha g$ is $q_\alpha$-excessive and that $(e^{-t\alpha} T_t)_{t > 0}$ is the semigroup of $q_\alpha$, then apply Lemma~\ref{lemma:characterization of excessive functions}). According to Lemma~\ref{lemma:main excessive} applied to $p = 2$ and $T_t = (T_{t/2})^*T_{t/2}$, there exist $C_\alpha > 0$ such that $h = \liminf_{\alpha \to 0+} C_\alpha f_\alpha$ is strictly positive and in $L^0(X,\mu)$. Using Fatou's lemma for positivity preserving operators we obtain 
 $$T_t h \leq \liminf_{\alpha \to 0+} T_t C_\alpha f_\alpha \leq \liminf_{\alpha \to 0+}  e^{t\alpha}  C_\alpha f_\alpha = h.$$
 Since this is true for any $t > 0$, Lemma~\ref{lemma:characterization of excessive functions} yields that $h$ is excessive. 
\end{proof}
\begin{remark}
The existence of an integral kernel for the semigroup (the heat kernel) is e.g. guaranteed if $X$ is a separable metric space, $\mu$ is a Borel measure of full support on $X$ and $T_t L^2(X,\mu) \subseteq C(X)$, $t >0$, see e.g. \cite{KLVW}. This property is a question of local regularity for solutions to the corresponding heat equation and satisfied in the situations discussed in \cite{PT,Tak1,KPP}.  Another criterion ensuring the existence of heat kernels is $L^1$-$L^\infty$ ultracontractivity in the case of Dirichlet forms.
\end{remark}

% 
% 
% 
% \begin{remark}
%  There are examples which show that such an $h$ need not  belong to $L^p(X,\mu)$. 
% \end{remark}

We now prove Lemma~\ref{lemma:main excessive} for operators  $T$ which are {\em positivity improving} ($f \geq 0$ and $f \neq 0$ implies that $Tf$ is strictly positive)  and satisfy a weak Harnack principle.  Then we show that the weak Harnack principle holds for kernel operators with strictly positive kernel. We start with two ergodicity properties for postivity improving operators.

\begin{lemma}[Ergodicity] \label{lemma:ergodicity}
 Let $1 < p < \infty$ and let $T \colon L^p(X,\mu) \to L^p(Y,\nu)$ be a positivity improving operator. Every  measurable $A \subseteq X$ with 
 $$T^* (T 1_A f)^{p-1} \leq 1_A T^* (T f)^{p-1}$$
 for all $f \in L^p_+(X,\mu)$ satisfies $\mu(A) = 0$ or $\mu(X \setminus A) = 0$.
 \end{lemma}
\begin{proof}
If $T$ is positivity improving, then $T^*$ is also positivity improving. If $\mu(A) \neq 0$, we can choose $f \in L^p_+(X,\mu)$ with $f \neq 0$ on $A$. The positivity improving property for $T$ and $T^*$ and the inequality for $A$ yields  
$$0 < T^* (T 1_A f)^{p-1} \leq 1_A T^* (T f)^{p-1}.$$
Hence, $\mu(X \setminus A) = 0$.
\end{proof}

\begin{lemma} \label{lemma:postivity excessive}
 Let $1 < p < \infty$ and let $T \colon L^p(X,\mu) \to L^p(Y,\nu)$ be positivity improving. Then $\lambda(T) > 0$.  Let $h \in L^+(X,\mu)$ satisfy 
 $$T^* (T h)^{p-1} \leq K h^{p-1} $$
 for some $K \geq \lambda(T) > 0$.  Then either $h = 0$, $0 < h < \infty$ or $h = \infty$ holds $\mu$-a.s.
 \end{lemma}
 \begin{proof} $\lambda(T) > 0$ directly follows from $T$ (and hence $T^*$) being positivity improving. 
 
 We show that the sets $A_1 = \{h > 0\}$ and $A_2 = \{h = \infty\}$ satisfy 
  $$T^* (T 1_{A_i} f)^{p-1} \leq 1_{A_i} T^* (T f)^{p-1}$$
  for all $f \in L^p_+(X,\mu)$, $i = 1,2$. With this at hand, the statement follows from Lemma~\ref{lemma:ergodicity}. 
  
For any $f \in L^p_+(X,\mu)$ we have 
  $$1_{\{h > 0\}} f = \lim_{n \to \infty}  f \wedge (nh) \text{ and } 1_{\{h = \infty\}} f = \lim_{n \to \infty} f \wedge (n^{-1}h),$$
  where the limits hold in $L^p(X,\mu)$ due to Lebesgue's dominated convergence theorem. Moreover, since $T$ and $T^*$ are positivity preserving, we obtain 
  $$T^* (T (f \wedge (nh)))^{p-1}  \leq T^* (T f)^{p-1} $$
  and 
  $$T^* (T (f \wedge (nh)))^{p-1}  \leq T^* (T (nh))^{p-1} \leq  K (nh)^{p-1}.  $$
  These observations yield
  \begin{align*}
   T^* (T 1_{\{h > 0\}} f)^{p-1} &= \lim_{n \to \infty}  T^* (T f \wedge (nh))^{p-1} \\
   &\leq \lim_{n \to \infty} (K (nh)^{p-1}) \wedge   T^* (T f)^{p-1}\\
   &= 1_{\{h > 0\}} T^* (T f)^{p-1}.
   \end{align*}
   A similar computation yields the statement for the set $\{h = \infty\}$.
 \end{proof}
 
%  
%  \begin{remark}
%   If $h \neq 0$ we could have directly deduced $h > 0$ from the positivity improving property alone. We included the above proof because maybe  weaker assumptions than $T$ being positivity improving could lead to the assertion of Lemma~1.2. 
%  \end{remark}

\begin{definition}[Weak Harnack principle]
Let $1 < p < \infty$ and let  $T \colon L^p(X,\mu) \to L^p(Y,\nu)$  be a positivity preserving linear operator. We say that $T$ satisfies the {\em weak Harnack principle at $\lambda \geq 0$} if  there exist $D > 0$ and measurable sets $A,B \subseteq X$ with $\mu(A), \mu(B) > 0$, such that for every $f\in L^p_+(X,\mu)$ with 
$$T^* (T f)^{p-1} \leq \lambda f^{p-1}$$
we have
$$\int_A f  d\mu \leq D \essinf\limits_{x \in B} f. $$
\end{definition}

We can now prove a version of Lemma~\ref{lemma:main excessive} for operators satisfying a weak Harnack principle.

\begin{lemma} \label{lemma:main}
  Let $1 < p < \infty$ and let $T \colon L^p(X,\mu) \to L^p(Y,\nu)$ be a positivity improving operator that satisfies the weak Harnack principle at some $\lambda \geq \lambda(T)$. Let $f_n \in L^p(X,\mu)$ be strictly positive with $T^* (T f_n)^{p-1} \leq \lambda  f_n^{p-1}$. Then there exist $C_n > 0$ such that $h = \liminf_{n \to \infty} C_n f_n$ satisfies $0 < h < \infty$ $\mu$-a.e. and 
  $$T^* (T h)^{p-1} \leq \lambda  h^{p-1}.$$
\end{lemma}

\begin{proof}
 Let $A,B \subseteq X$ and let $D > 0$ as in the weak Harnack principle, which  we apply to $f_n$. From $\int_A f_n d\mu > 0$ we obtain $\essinf_{x \in B} f_n > 0$. Hence, we can choose $C_n > 0$ such that $\essinf_{x \in B} C_n f_n = 1$. We set
 $$h = \liminf_{n \to \infty} C_n f_n \in L^+(M). $$
 Fatou's lemma for positivity preserving operators yields 
 $$T^* (T h)^{p-1} \leq \liminf_{n \to \infty} T^* (T C_n f_n)^{p-1} \leq \liminf_{n\to \infty} \lambda (C_n f_n)^{p-1} = \lambda h^{p-1}. $$
 According to Lemma~\ref{lemma:postivity excessive}, it suffices to exclude the cases $h = 0$ and $h = \infty$ to obtain the desired statement. From $\essinf_{x \in B} C_n f_n = 1$ we infer $h \geq 1$ on $B$. Therefore, $h \neq 0$. Moreover, Fatou's Lemma yields
 $$\int_A h d\mu \leq \liminf_{n \to \infty}\int_A C_n f_n d\mu  \leq  \liminf_{n \to \infty} D \essinf_{x \in B} C_n f_n = D. $$
 In particular, $h$ is $\mu$-a.s. finite on $A$ and so we obtain $h \neq \infty$. 
\end{proof}

% 
% 
% 
% \begin{lemma}
%  Let $1 < p < \infty$ and let $T \colon L^p(X,\mu) \to L^p(Y,\nu)$ be a positivity improving operator that satisfies the weak Harnack principle for some $C > \lambda(T)$. Then there exists $0 < h \in L^0(X,\mu)$ such that 
%  %
%  $$T^* (T h)^{p-1} \leq \lambda(T) h^{p-1}.$$
%  %
% \end{lemma}
% \begin{proof}
%  Let $K_n$ be a decreasing sequence converging to $\lambda(T)$ and let $0 \neq f_n \in L^p_+(X,\mu)$ be corresponding functions with 
%  %
%  $$T^* (T f_n)^{p-1} \leq K_n (f_n)^{p-1}.$$
%  %
%  
% \end{proof}

\begin{lemma}
 Let $k \colon Y \times X \to (0,\infty)$ be bimeasurable. Then the function 
 $$\tilde{k} \colon X \times X \to [0,\infty),\quad \tilde k(x,y) = \int_Y k(z,x) k(z,y) d\nu (z)$$
 is bimeasurable and there exists $c > 0$ and a measurable set  $A \subseteq X$ with $\mu(A) > 0$ such that 
 $$\tilde k \geq c \text{ on } A \times A.$$
\end{lemma}
\begin{proof} The measurability of $\tilde k$ follows from Fubini's theorem.  For the estimate we can assume that $\mu$ and $\nu$ are probability measures (else restrict everything to sets of finite measure and rescale).  For $n \in \N$ we consider the sets $D_n = \{(y,x) \in Y \times X \mid k(y,x) \geq 1/n\}$ and for $x \in X$ we consider the measurable sections $(D_n)^x = \{y \in Y \mid (y,x) \in D_n\}$. 
 By definition   we have
 $$\tilde k (x,y) = \int_Y k(z,y)k(z,x) d\nu(z) \geq \frac{1}{n^2} \nu((D_n)^x \cap (D_n)^y). $$
 Hence, we need to prove that for some $n \in \N$ there exists $C > 0$ and a set $A$ with $\mu(A) > 0$ and $\nu((D_n)^x \cap (D_n)^y) \geq C > 0$ for all $x,y \in A$.

Our assumption $k > 0$ yields $D_n  \nearrow Y \times X$, as $n \to \infty$. Using the monotone convergence theorem, for $\varepsilon > 0$ we find $n_\varepsilon \in \N$ with 
 $$(\nu \otimes \mu) (D_{n_\varepsilon}) \geq 1 - \varepsilon.$$
 We claim that $A_\varepsilon = \{x \in X \mid \nu((D_{n_\varepsilon})^x) > 1 - 2\varepsilon\}$ satisfies  $\mu(A_\varepsilon) > 0$ (here we use Fubini's theorem for the measurability of $A_\varepsilon$). If this were not the case, we would have $ \nu((D_{n_\varepsilon})^x) \leq 1 - 2\varepsilon$ for $\mu$-a.e. $x \in X$. But then Fubini's theorem  leads to the contradiction
 $$(\nu \otimes \mu) (D_{n_\varepsilon}) = \int_Y \nu ((D_{n_\varepsilon})^x) d\mu(x) \leq 1 - 2\varepsilon.$$
 Hence, for $x,y \in A_\varepsilon$ we have 
 $$\nu((D_{n_\varepsilon})^x \cap (D_{n_\varepsilon})^y) =  \nu((D_{n_\varepsilon})^x) + \nu((D_{n_\varepsilon})^y) - \nu((D_{n_\varepsilon})^x \cup (D_{n_\varepsilon})^y) > 1 - 4 \varepsilon,  $$
 so that $A = A_\varepsilon$ for some $\varepsilon < 1/4$ is a set as required.
\end{proof}

\begin{remark}
 If there exists $c > 0$ and not neglibile sets $A \subseteq X,B \subseteq Y$ such that $k \geq c$ on $B \times A$, then the statement of the previous lemma trivially follows from the definition of $\tilde k$. Note however that $k > 0$ does not imply the existence of such sets, see e.g. \cite[Theorem~B.1]{HS}.
\end{remark}

\begin{lemma}[Weak Harnack principle for kernel operators]\label{lemma:weak harnack kernel}
Let $1 < p < \infty$ and let $T \colon L^p(X,\mu) \to L^p(Y,\nu)$ be a kernel operator with strictly positive kernel. Then $T$ satisfies the weak harnack principle at every $\lambda > 0$. 
\end{lemma}
\begin{proof}
 Let $\lambda > 0$ and let $f \in L^p_+(X,\mu)$ with $T^* (T f)^{p-1} \leq \lambda f^{p-1}$. Using the previous lemma we choose $c > 0$ and a set $A$ with $\mu(A) > 0$ such that 
  $$\tilde k (x,y) = \int_Y k(z,x)k(z,y) d\nu(z) \geq c$$
 for all $x,y \in A$. Let $q$ such that $1/q + q/p = 1$ and let $\varphi \in L_+^q(X,\mu)$. Using Minkowski's inequality for integrals and Fubini's theorem we obtain
\begin{align*}
 \lambda \int_X  f(x) \varphi(x) d\mu (x) &\geq   \int_X \varphi(x) \left( \int_Y k(z,x) \left( \int_X k(z,y)f(y) d\mu(y)  \right)^{p-1} d\nu(z) \right)^{1/(p-1)} d\mu(x)\\
 &\geq  \left(\int_X \varphi(x)  \left(  \int_Y k(z,x)   \int_X k(z,y)f(y) d\mu(y)    d\nu(z) \right)^{p-1} d\mu(x)\right)^{1/(p-1)}\\
 &=   \left(\int_X \varphi(x)   \left(  \int_X \tilde k(x,y)   f(y) d\mu(y)   \right)^{p-1} d\mu(x)\right)^{1/(p-1)}\\
 &\geq c   \int_A \varphi(x)  d\mu(x) \int_A f(y) d\mu(y).
\end{align*}
 Now we let $\varphi  = 1_B$, where $B \subseteq A$ is a set wit $0 < m(B) < \infty$ such that $f \leq \essinf_{x \in A} f + \varepsilon$ on $B$. We obtain 
 $$\lambda(\essinf_{x \in A} f + \varepsilon) m(B) \geq c m(B) \int_A f(y) d\mu(y).  $$
 Since $\varepsilon > 0$ is arbitrary, this proves the claim with constant $D = \lambda/c > 0$.
\end{proof}
\begin{remark}
 We proved that in principle the weak Harnack principle holds at all $\lambda > 0$. However, nontrivial functions  with $T^* (T f)^{p-1} \leq \lambda f^{p-1}$ only exist for $\lambda \geq \lambda(T)$.
\end{remark}
% 
% 
% \begin{proof}[Proof of Theorem~\ref{theorem:main}] Kernel operators with strictly positive kernel are certainly positivity improving. Hence, the statement follows from Lemma~\ref{lemma:main} and weak Harnack principle for kernel operators Lemma~\ref{lemma:weak harnack kernel}.
%  
% \end{proof}

\begin{proof}[Proof of Lemma~\ref{lemma:main excessive}]
 This follows directly from the previous two lemmas.
\end{proof}

\bibliographystyle{plain}
 
\bibliography{literatur}

\end{document}